\documentclass[a4paper, reqno, 11pt]{amsart}
\usepackage{amssymb,amsmath,amsthm,amsfonts,color} 

\usepackage[left=3.5cm,right=3.5cm,top=2.5cm,bottom=2.5cm]{geometry}

\usepackage[colorlinks=true, pdfstartview=FitV, linkcolor=blue, 
            citecolor=blue, urlcolor=blue]{hyperref}

\numberwithin{equation}{section}

\newcommand{\black}{\color{black}}

\newcommand{\cE}{{\mathcal E}}

\newcommand{\cC}{\mathcal C}
\newcommand{\cB}{\mathcal B}
\newcommand{\cA}{\mathcal A}

\newcommand{\cF}{\mathcal{F}}

\newcommand{\cH}{\mathcal{H}}
\newcommand{\cN}{\mathcal{N}}
\newcommand{\cL}{\mathcal{L}}

\newcommand{\cG}{\mathcal{G}}

\newcommand{\K}{\mathbb{O}}
\newcommand{\cM}{{\mathcal M}}

\newcommand{\N}{\mathbb{N}}

\newcommand{\R}{\mathbb{R}}

\renewcommand{\S}{\mathbb{S}}
\renewcommand{\P}{\mathbb{P}}
 
\newcommand{\hull}[1]{\mathrm{co}(#1)}

\newcommand{\functwo}{\mathfrak{F}_2}
\newcommand{\func}{\mathfrak{F}}

\newcommand{\sdist}{\tilde{\mathrm d}} 
\newcommand{\dist}{\mathrm{d}}  

\newcommand{\force}{\Upsilon}

\newcommand{\loc}{\mathrm{loc}}
\newcommand{\no}{\nonumber}
\newcommand{\strictlyincluded}{\subset\subset}

\newcommand{\res}{\mathop{\hbox{\vrule height 7pt width 0.5pt depth 0pt
\vrule height 0.5pt width 6pt depth 0pt}}\nolimits}

\newcommand{\argmin}{\mathop{\mathrm{argmin}}\,\,}

\newcommand{\p}{\partial}

\newcommand{\Per}{\mathrm{Per}}

\newcommand{\diam}{\mathrm{diam}}

\renewcommand{\hat}{\widehat}
\renewcommand{\tilde}{\widetilde}
\newcommand{\twofase}[4]{F_{#1}^{#2} (#3,#4)}

\theoremstyle{plain}
\newtheorem{theorem}{Theorem}[section]
\newtheorem{lemma}[theorem]{Lemma}
\newtheorem{definition}[theorem]{Definition}
\newtheorem{proposition}[theorem]{Proposition}
\newtheorem{corollary}[theorem]{Corollary}

\theoremstyle{definition}

\newtheorem{remark}[theorem]{Remark}

\date{\today}

\begin{document}
\title[]{Minimizing movements for forced anisotropic  mean curvature flow of partitions with mobilities}

\author{Giovanni Bellettini, Antonin Chambolle, Shokhrukh Kholmatov}

\address[G. Bellettini]{University of Siena, via Roma 56,  53100 Siena, Italy \&  International Centre for Theoretical Physics (ICTP), Strada Costiera 11, 34151 Trieste, Italy}
\email{bellettini@diism.unisi.it}

\address[A. Chambolle]{CMAP, Ecole Polytechnique,
91128 Palaiseau Cedex, France}
\email{antonin.chambolle@cmap.polytechnique.fr}

\address[Sh. Kholmatov]{University of Vienna,
Oskar-Morgenstern-Platz 1, 1090  Vienna, Austria}
\email{shokhrukh.kholmatov@univie.ac.at}

\begin{abstract}
Under suitable assumptions on the family of anisotropies, we prove the existence of a weak global $\frac{1}{n+1}$-H\"older continuous in time mean curvature flow with mobilities of a bounded anisotropic partition  in any dimension using the method of minimizing movements.  The result is extended to the case  when suitable driving forces are present. We improve the H\"older exponent to $\frac12$ in the case of partitions with the same anisotropy and the same mobility and provide a weak comparision result in this setting for  a weak anisotropic mean curvature flow of a partition and an anisotropic mean curvature two-phase flow.
\end{abstract}


\maketitle



\section{Introduction}

Many processes in material sciences such as phase transformation, crystal growth,  grain growth,  stress-driven rearrangement instabilities, {\it etc.} can be modelled as geometric interface motions, in which surface tensions act  as a principal driving force  (see 
e.g., \cite{Mullins:1956,Brakke:1978,TCH:1992,KL:2001} and references therein). A typical example of such a motion is anisotropic mean curvature flow: given a norm $\phi$ on $\R^n$ (called anisotropy), the equation for the anisotropic mean curvature flow of hypersurfaces parametrized as $\Gamma_t$ reads as 
\begin{equation}\label{mean_curvature_eq}
\beta(\nu)V = -{\rm div}_{\Gamma_t}[\nabla \phi(\nu)]\qquad \text{on $\Gamma_t,$} 
\end{equation}
where $V$ denotes the normal velocity of $\Gamma_t$ in the direction of the unit outer normal $\nu$ of $\Gamma_t$ and  $\beta $  is the mobility, a positive kinetic  coefficient \cite{GPR:2001}.  Anisotropic mean curvature flow is called crystalline provided the boundary of the Wulff shape  $W_\phi:=\{\phi\leq 1\}$ lies on finitely many hyperplanes; 
in this quite interesting case,  
equation \eqref{mean_curvature_eq} must be properly interpreted, 
due to the nondifferentiability of $\phi$; see for instance 
\cite{AT:1995,GG:1996,Taylor:1999,GG:2000,BNP:2001,BNP:2001.2,GP:2016,ChMP:2017,GP:2018,ChMNP:2019.ams,ChMNP:2019.apde}. 
Equation \eqref{mean_curvature_eq} (sometimes referred to as the two-phase evolution) can be generalized to the case of networks in the 
plane, and more generally to the case of partitions of space (sometimes called the multiphase case): here  the evolving sets are intrisically nonsmooth, since the presence of triple junctions (in the plane), or multiple lines, quadruple points etc. (in space)
during the flow is unavoidable. It must be stressed that evolutions of partitions received recently a lot of attention from the mathematical community 
\cite{Taylor:1993,Freire:2010.analpde,Freire:2010commpde,DGK:2014,MNPS:2016,KT:2017,INSh:2019,SchW:2017} both as a natural generalization of the case of two phases, and because they model a variety of physical phenomena, such as grain growth  and evolution of multicrystals \cite{KL:2001,BRN:2003}.

The presence of singularities at finite time is a common feature of mean curvature flow type motions, 
both in the two-phase case 
\cite{Grayson:1989,Huisken:1990,Huisken:1993,Huisken:1998,Mantegazza:2011}, and in the multiphase case  (see for instance \cite{MNPS:2016}). This phenomenon justifies to introduce
and study some notion of weak solution, defined globally in time. 
This has been done in several different ways: 
just to quote a few, 
the Brakke varifold-solution  \cite{Brakke:1978}, 
the viscosity solution  (see \cite{Giga:2006} and references therein), 
the Ilmanen elliptic regularization \cite{Ilmanen:1994},  the level-set theoretic subsolution and the minimal  barrier solution   (see \cite{Bellettini:2012} and references therein), the Almgren-Taylor-Wang  \cite{ATW:1993} and  Luckhaus-Sturzenhecker \cite{LS:1995} solutions, next included by De Giorgi into his notion of minimizing movement and generalized minimizing movement (GMM) \cite{Degiorgi:1993,Degiorgi:1996};  
see also 
\cite{ESS:1992,MR:2008}.
Some of those solutions (e.g., the Brakke solution \cite{Brakke:1978,Tonegawa:2019}, the GMM solution \cite{BH:2018.siam}, the elliptic regularization 
\cite{SchW:2017}) can be adapted to treat the multiphase case at least in the Euclidean case, especially those that do not rely heavily on the comparison principle. Also, the existence of a distributional solution  of mean curvature evolution of partitions on the torus using the time thresholding method introduced in \cite{MBO:1992}  has been proved in \cite{LO:2016}; see also \cite{KT:2017}.

The aim of the present paper is to prove the existence of a 
GMM for anisotropic mean curvature flow of partitions with no restrictions on the space dimension, in the presence of a set of mobilities and forcing terms,
and to point out 
some qualitative properties of this weak evolution, which are
obtained via a comparison argument with a
GMM of each single phase considered separately. 

Let us recall the definition of GMM for partitions from 
\cite{Degiorgi:1996} (see Definition \ref{def:g_partitions}
for the notion of bounded partition).

\begin{definition}[\textbf{Generalized minimizing movement for partitions}]\label{def:GMM}
Let $\P_b(N+1)$ be the set of all bounded $(N+1)$-partitions of $\R^n$ (Definition \ref{def:g_partitions}) endowed with the $L^1(\R^n)$-convergence, and let
$\func: \P_b(N+1)\times \P_b(N+1)\times [1,+\infty) \to[-\infty,+\infty]$ be defined as 
$$
\func(\cA,\cB,\lambda) = \sum\limits_{j=1}^{N+1} P_{\phi_j}(A_j) + \lambda \sum\limits_{j=1}^{N+1} \int_{A_j\Delta B_j} 
\dist_{\psi_j}(x,\p B_j)dx +\sum\limits_{j=1}^{N+1} \int_{A_j} H_jdx, 
$$
where $\phi_j$ and $\psi_j$ are norms 
on $\R^n,$  called anisotropies and mobilities, respectively,  $H_i\in L^1_\loc(\R^n),$ $i=1,\ldots,N,$ and $H_{N+1}\in L^1(\R^n)$ are 
driving forces,  $P_{\phi_j}(A_j)$ is the $\phi_j$-anisotropic perimeter, $\cA=(A_1,\ldots,A_{N+1}),$  $\cB=(B_1,\ldots,B_{N+1})$ and  $\dist_{\psi_j}(\cdot,E)$ is the $\psi_j$-distance function from 
$E\subseteq\R^n.$   We say that a map $\cM:[0,+\infty)\to \P_b(N+1)$ 
is a GMM
associated to $\func$  starting from $\cG\in \P_b(N+1)$,  and we write $\cM\in GMM(\func,\cG),$ if there exist
$\cL:[1,+\infty)\times \N_0 \to \P_b(N+1)$ and a diverging sequence $\{\lambda_h\}$ such that
$$
\lim\limits_{h\to+\infty} \cL(\lambda_h,[\lambda_ht]) = \cM(t)\quad \text{in $L^1(\R^n)$  for any $t\ge0,$}
$$
where the bounded partitions $\cL(\lambda,k),$ $\lambda\ge1,$ $k\in\N_0,$ are defined  inductively as $\cL(\lambda,0)=\cG$  and
$$
\func(\cL(\lambda,k+1),\cL(\lambda,k),\lambda) = 
\min\limits_{\cA\in \P_b(N+1)} \func(\cA,\cL(\lambda,k),\lambda)
\qquad \forall k \geq 0.
$$
\end{definition}

Our first result (see Theorems \ref{teo:existence_GMM} and \ref{teo:existence_GMM2f} for the precise  statements) 
extends the existence results of \cite{BH:2018.siam} to the case with anisotropies, mobilities and external forces. We also improve the $\frac{1}{n+1}$-H\"older regularity in time of GMM proven in \cite{BH:2018.siam} to $1/2$-H\"older continuity in the two-phase case, without any restriction on the anisotropies.

\begin{theorem}\label{teo:main_intro1}
Suppose that the driving forces $\{H_i\}$ satisfy \eqref{H_condition}.
Let  $\cG\in \P_b(N+1)$.  
The following assertions hold:
\begin{itemize}
\item[(a)] Let $N\ge2.$ If $\{\phi_j\}$ 
satisfy \eqref{kappaning_def} and \eqref{Phi_condition}, 
then $GMM(\func,\cG)$ is nonempty.
Moreover, any $\cM = (M_1,\ldots, M_{N+1})\in GMM(\func,\cG)$ is locally $\frac{1}{n+1}$-H\"older continuous in time and 
for any $t\ge 0,$ $\bigcup\limits_{j=1}^N M_j(t)$ is contained in the bounded closed convex set related to  $\cG$ and $H_j$ (see \eqref{dddid}).

\item[(b)] Let $N=1.$ Then, with no assumptions 
on the anisotropies $\phi_1,\phi_2$ and the mobilities $\psi_1,\psi_2,$  $GMM(\func,\cG)$ is non-empty. Moreover, any $\cM\in GMM(\func,\cG)$ is locally $1/2$-H\"older continuous in time. 
\end{itemize}
\end{theorem}

To prove  Theorem  \ref{teo:main_intro1} we establish  uniform density estimates for  minimizers of $\func$  using the method of cutting out and filling in with balls, an argument of \cite{LS:1995}. At this level the presence of mobilities does not create any new substantial problem. 
While in the two-phase case we do not need any assumption on the anisotropies, 
in the multiphase case assumption \eqref{kappaning_def} is needed to get the lower volume density estimate for minimizers which is important in the proof of time-continuity of GMM. 

In case of partitions with the same anisotropies and the same mobilities and without forcing,  the H\"older exponent of GMM can be improved to $1/2$ (see Theorem \ref{teo:existence_GMM_half}). Denoting by $\functwo$ the restriction of $\func$ to two-phase case without forcing (see \eqref{functo2}), this can be done using the comparison property (Theorem \ref{teo:comparison}) between the minimizers of $\func$ and the minimizers of $\functwo$.  This comparison result also 
enables us to get a weak comparison flow of corresponding multiphase 
and two-phase flows (Theorem \ref{teo:weak_comparison}): 

\begin{theorem}\label{teo:main_intro2}
Assume that $\phi_j=\phi_i$ and $\psi_j=\psi_i$ for all $i,j=1,\ldots,N+1,$ and $H_i=0$ for all $i=1,\ldots,N+1.$ Then any $\cM \in GMM(\func,\cG)$ is locally $1/2$-H\"older continuous in time and $\bigcup\limits_{j=1}^N M_j(t)$ is contained in the closed convex envelope of the union $\bigcup\limits_{j=1}^N G_j$ of the bounded components of $\cG$ for any $t\ge0.$  Moreover:

\begin{itemize}
\item[(a)] for any $\cM\in GMM(\func,\cG)$  and for any $i\in\{1,\ldots,N+1\}$ there exists $L_i\in GMM(\functwo,G_i)$  such that 
$$
L_i(t)\subseteq M_i(t),\qquad t\ge0;
$$

\item[(b)] Let $C_i,$ $i\in\{1,\ldots,N\},$ and $C_{N+1}$ be any convex sets such that $C_i\subset G_i$ for any $i\in\{1,\ldots,N+1\}$ and let $L_i\in MM(\functwo,C_i)$ be the unique minimizing movement starting from $C_i.$ Then for any $\cM\in GMM(\func,\cG),$ 
\begin{equation}\label{appear}
L_i(t)\ne\emptyset \quad \Longrightarrow \quad L_i(t)\ne \emptyset\qquad\text{for any $i\in\{1,\ldots,N\},$}
\end{equation}
and 
\begin{equation}\label{disappear}
\R^n\setminus L_{N+1}(t)=\emptyset \quad \Longrightarrow \quad \R^n\setminus M_{N+1}(t)= \emptyset. 
\end{equation}
\end{itemize}
\end{theorem}

Note that the comparison principle \eqref{disappear} implies that any bounded partition will disappear in the long run; moreover, \eqref{appear} allows to estimate the extinction time
of the $i$-th bounded phase (Corollary \ref{cor:time_est}).

Finally, let us mention that a natural problem remains open, namely 
the consistency of GMM with the classical solution, 
provided the latter exists, at least on a short time interval. 
Such a result has been proven by 
Almgren-Taylor-Wang in \cite{ATW:1993} 
in the two-phase case without mobility; 
the proof is based on various stability properties of the flow, 
and using comparison arguments. 
It has also been proven by Almgren-Taylor \cite{AT:1995} in the two-phase crystalline case. 
However,  consistency is not known in the case of networks in the plane 
(and a fortiori for partitions in space), even in the Euclidean case without mobilities and forcing.

The paper is organized as follows. In Section \ref{sec:notation} we introduce the notation, some results from the theory of sets of finite perimeter, and  the definition of a partition. In Section \ref{sec:density_for_almost_min} we prove the density estimates for almost minimizers. The existence of generalized minimizing movements (Theorem \ref{teo:main_intro1}) is established in Section \ref{sec:existence_GMM_partition}. 
In Section \ref{sec:improve_holder} we improve the H\"older regularity of GMM (Theorem \ref{teo:main_intro2}) and provide some weak comparison principles.

\subsection*{Acknowledgments}
The first author acknowledges support from GNAMPA (Gruppo Nazionale per l'Analisi Matematica, la Probabilit\'a
e le loro Applicazioni) of INdAM (Istituto Nazionale di Alta Matematica). The third author acknowledges support from the Austrian Science Fund (FWF) project M~2571-N32.

\section{Notation and preliminaries}\label{sec:notation}

In this section we introduce the notation  and collect some important properties of sets of locally finite perimeter. The standard references for $BV$-functions and sets of finite perimeter are \cite{AFP:2000,Maggi:2012}.

We use $\N_0$ to denote the set of all nonnegative integers.   The symbol $B_r(x)$  stands for the open ball in  $\R^n$  centered at $x\in\R^n$  of radius $r>0.$ The characteristic function of a Lebesgue measurable set $F$ is denoted by $\chi_F$ and its Lebesgue measure by $|F|;$ we set also $\omega_n:=|B_1(0)|.$  We denote by $E^c$ the complement of $E$ in $\R^n.$

Given a norm $\psi$ in $\R^n$ and a nonempty set $E\subseteq\R^n,$ $\dist_\psi(\cdot,E)$ stands for the $\psi$-distance  from $E,$ i.e.,
$$
\dist_\psi(x,E) = \inf\{\psi(x-y):\,\, y\in E\},
$$
and 
$$
\sdist_\psi(x,\p E) = \dist_\psi(x,E) - \dist_\psi(x,\R^n\setminus E)
$$
is the signed $\psi$-distance function from $\p E,$ negative inside $E.$ When $\psi$ is Euclidean for simplicity we drop the dependence on $\psi.$ We also write 
$$
\diam_\psi E:=\sup\{\psi(x-y):\,\,x,y\in E\}
$$ 
to denote the $\psi$-diameter of $E.$  

By $\K(\R^n)$ (resp. $\K_b(\R^n)$) we denote  the collection of all open (resp. open and bounded) subsets of $\R^n.$  The set of $L_\loc^1(\R^n)$-functions having locally bounded total variation in $\R^n$ is denoted by  $BV_\loc(\R^n)$ and the elements of 
$$
BV_\loc(\R^n,\{0,1\}):=\{E\subseteq\R^n:\,\, \chi_E\in BV_\loc(\R^n)\}
$$
are called locally finite perimeter sets. 
Given a $E\in BV_\loc(\R^n,\{0,1\})$ we denote by
\begin{itemize}
\item[a)]  $P(E,\Omega):=\int_\Omega|D\chi_E|$ the perimeter of $E$
in $\Omega\in \K(\R^n);$
\item[b)] $\p E$ the measure-theoretic boundary of $E:$
$$
\p E:=\{x\in \R^n:\,\, 0<|B_\rho\cap E|<|B_\rho|\quad\forall \rho>0\};
$$
\item[c)] $\p^*E$ the reduced boundary of $E;$
\item[d)] $\nu_E$   the outer generalized unit normal to $\p^*E.$
\end{itemize}
For simplicity, we set $P(E):=P(E,\R^n)$  provided $E\in BV(\R^n;\{0,1\}).$ Further, given a Lebesgue measurable set $E\subseteq\R^n$ and $\alpha\in [0,1]$ we define 
$$
E^{(\alpha)}:=\left\{x\in\R^n:\,\,\lim\limits_{\rho\to0^+}
\frac{|B_\rho(x)\cap E|}{|B_\rho(x)|} =\alpha\right\}.
$$
Unless otherwise stated, we always suppose that any locally finite perimeter set $E$ we consider coincides with $E^{(1)}$ (so that by \cite[Eq. 15.3]{Maggi:2012} $\p E$ coincides with the topological boundary). We recall that $\overline{\p^*E}=\p E$  and  $D\chi_E = \nu_E d\cH^{n-1}\res \p^*E,$  where $\cH^{n-1}$ is the $(n-1)$-dimensional Hausdorff measure in $\R^n$ and $\res$ is the symbol of restriction.

\begin{remark}
Given $E\in BV_\loc(\R^n;\{0,1\})$ the map  $\Omega\in \K(\R^n)\mapsto P(E,\Omega)$ extends to a Borel measure in $\R^n$  so that $P(E,B) = \cH^{n-1}(B \cap \p^*E)$ for every Borel set $B \subseteq\R^n.$ Moreover, by \cite[Theorem 3.61]{AFP:2000} for every  $E\in BV_\loc(\R^n;\{0,1\})$ 
$$
\cH^{n-1}(\R^n\setminus (E^{(0)} \cup E  \cup \p^*E)) = 0.
$$
In particular, $\cH^{n-1}(E^{(1/2)}\setminus \p^*E)=0.$
\end{remark}

\begin{theorem}\cite[Theorem 16.3]{Maggi:2012}
If $E$ and $F$ are sets of locally finite perimeter, and we let
\begin{align*}
\{\nu_E = \nu_F \} = & \{x \in \p^*E \cap \p^*F:\,\, \nu_E(x) = \nu_F(x)\},\\
\{\nu_E = -\nu_F \} = & \{x \in \p^*E \cap \p^*F:\,\, \nu_E(x) = - \nu_F(x)\},
\end{align*}
then $E \cap F,$ $E \setminus F$ and $E \cup F$ are locally finite perimeter sets  with
\begin{align}\p^*(E \cap F) \approx & \big(F  \cap \p^* E\big) \cup 
\big(E  \cap \p^* F\big) \cup \big\{\nu_E = \nu_F\big\}, \label{ess_intersection}\\
\p^*(E \setminus F) \approx & \big(F^{(0)} \cap \p^* E\big) \cup 
\big(E  \cap \p^* F\big) \cup \big\{\nu_E = -\nu_F\big\}, \label{ess_differense}\\
\p^*(E \cup F) \approx & \big(F^{(0)} \cap \p^* E\big) \cup 
\big(E^{(0)} \cap \p^* F\big) \cup \big\{\nu_E = \nu_F\big\}, \label{ess_union}
\end{align}
where $A\approx B$  means   $\cH^{n-1}(A\Delta B) =0.$
%
%
%
\end{theorem}
\smallskip

The following generalizes the notion of the perimeter. 
 
\begin{definition}[\textbf{Anisotropic perimeter}]
Let $\phi:\R^n\to[0,\infty)$ be a norm in $\R^n.$ Given $\Omega\in \K(\R^n)$ the $\phi$-perimeter  of  $E\in BV(\Omega;\{0,1\})$ is 
$$
P_\phi(E,\Omega):=\int_{\Omega\cap\p^*E} \phi(\nu_E)d\cH^{n-1}.
$$ 
When $\Omega=\R^n,$ we write $P_\phi(E):=P_\phi(E,\R^n),$ and when $\phi$ is Euclidean, we write $P$ in place of $P_\phi.$
\end{definition}

It is well-known that $E\mapsto P_\phi(E;\Omega)$ is $L^1_\loc(\Omega)$-lower semicontinuous. 
Recall also that for every $E,F\in BV_\loc(\R^n;\{0,1\})$ and $\Omega\in \K(\R^n)$
\begin{equation}\label{famfor}
P_\phi(E\cap F,\Omega) + P_\phi(E\cup F,\Omega) \le P_\phi(E,\Omega) + P_\phi(F,\Omega).
\end{equation}

\subsection{Anisotropic partitions}\label{sec:partitions}

We recall the notions of partition, almost-minimizer  and bounded partition, see \cite{BH:2018.siam}. 

\begin{definition}[\bf Partition]\label{def:partitions}
Given an integer  $N\ge2,$  an $N$-tuple  $\cC=(C_1,\ldots,C_{N})$ of subsets of  $\R^n$ is called an   $N$-partition   of $\R^n$ (a partition, for short) if 
\begin{itemize}
\item[(a)] $C_i\in BV_\loc(\R^n;\{0,1\})$ for every $i=1,\ldots, N,$
\item[(b)] $\sum\limits_{i=1}^{N} |C_i\cap K| =|K|$  for each compact $K\subset \R^n.$
\end{itemize}
\end{definition}

The collection of all $N$-partitions of $\R^n$  is denoted by $\P(N).$ Our assumptions $C_i=C_i^{(1)}$ imply $C_i\cap C_j=\emptyset$ for $i\ne j.$  

The elements of $\P(N)$ are denoted by calligraphic letters  $\cA,\cB,\cC,\ldots$ and the components of $\cA\in \P(N)$ by the corresponding roman letters 
$(A_1,\ldots,A_N).$ 

Let $\phi_1,\ldots,\phi_N$ be norms in $\R^n$ and set $\Phi:=\{\phi_1,\ldots,\phi_N\}.$ The functional 
$$
(\cA,\Omega)\in \P(N)\times \K(\R^n)\mapsto \Per_\Phi(\cA,\Omega):= \sum\limits_{i=1}^{N} P_{\phi_i}(A_i,\Omega)
$$
is called the anisotropic perimeter, or $\Phi$-perimeter of the partition $\cA$ in $\Omega.$ For simplicity, we write $\Per_\Phi(\cA):=\Per_\Phi(\cA,\R^n).$ For shortness, we also set $\Per_\Phi=\Per$ when all $\phi_i$ are Euclidean. 
Since $N$ is finite, there exist $0<c_\Phi \le C_\Phi<+\infty $ such that 
\begin{equation}\label{Phi_assump}
c_\Phi\le \phi_i(\nu)\le C_\Phi 
\end{equation}
for any $i=1,\ldots,N$ and $\nu\in\S^{n-1},$
therefore,
\begin{equation}\label{per_vs_anis_per}
c_\Phi \Per(\cA,\Omega) \le \Per_\Phi(\cA,\Omega) \le C_\Phi \Per(\cA,\Omega). 
\end{equation}
In view of \cite[Proposition 3.3]{BH:2018.siam}
\begin{align*}
\Per_\Phi(\cA,\Omega)= \sum\limits_{1\le i<j\le N}\,\, \int_{\Omega\cap \p^*A_i\cap \p^*A_j} \big(\phi_i(\nu_{A_i}) + \phi_j(\nu_{A_i})\big)\,d\cH^{n-1},
\end{align*}
i.e., on a generalized hypersurface $\Sigma_{ij}:=\Omega\cap \p^*A_i\cap \p^*A_j$ dividing the phase $i$ from the phase $j$ the perimeter contributes 
$$
\int_{\Sigma_{ij}} (\phi_i(\nu_{\Sigma_{ij}}) +\phi_j(\nu_{\Sigma_{ij}})) d\cH^{n-1},
$$  
where $\nu_{\Sigma_{ij}}$ is the generalized unit normal to $\Sigma_{ij}$ pointing for instance from $A_i$ to $A_j.$
We set 
\begin{equation}\label{L_1dist}
\cA\Delta\cB:= \bigcup\limits_{j=1}^{N} A_j\Delta B_j \qquad\text{and}\qquad |\cA\Delta \cB|:= \sum\limits_{j=1}^{N} |A_j\Delta B_j|, 
\end{equation}
where $\Delta $ is the symmetric difference of sets, i.e., 
$E\Delta F = (E\setminus F)\cup (F\setminus E).$

We say that the sequence $\{\cA^{(k)}\}\subseteq\P(N)$ 
{\it converges} to $\cA\in \P(N)$ in $L_\loc^1(\R^n)$  if 
$$
|(\cA^{(k)}\Delta \cA)\cap K|:= \sum\limits_{j=1}^{N} |(A_j^{(k)}\Delta A_j)\cap K| \to 0
\qquad\text{as $k\to+\infty$}
$$
for every compact set $K\subset \R^n.$
Since $E\in BV_\loc(\R^n;\{0,1\})\mapsto P_{\phi_i}(E,\Omega)$ is $L_\loc^1(\R^n)$-lower semicontinuous 
for any $\Omega\in \K(\R^n),$ so is the map $\cA\in \P(N)\mapsto \Per_\Phi(\cA,\Omega).$   From \eqref{per_vs_anis_per} and \cite[Theorem 3.2]{BH:2018.siam} we get
 
\begin{proposition}[\textbf{Compactness}] 
Let $\{\cA^{(l)}\}\subset \P(N)$  be a sequence of partitions such that 
$$
\sup\limits_{l\ge1} \,\,\Per_\Phi(\cA^{(l)},\Omega)<+\infty\qquad\forall \Omega\in\K_b(\R^n). 
$$
Then there exist a partition $\cA\in \P(N)$ and a subsequence $\{\cA^{(l_k)}\}$ converging  to $\cA$ in $L_\loc^1(\R^n)$ as $k\to+\infty.$ 
\end{proposition}

\subsection{Bounded partitions}

\begin{definition}[\textbf{Bounded partition}]\label {def:g_partitions}
A  partition $\cC=(C_1,\ldots,C_{N+1})\in\P(N+1)$ is called  bounded, and we write $\cC\in \P_b(N+1),$ if $C_i$ is bounded for each $i=1,\ldots,N.$ 
\end{definition}

Note that $\cA\Delta\cB \strictlyincluded \R^n$ for every $\cA,\cB\in \P_b(N+1),$ and therefore, 
$$
|\cA\Delta\cB| = \sum\limits_{j=1}^{N+1} |A_j\Delta B_j|
$$
is the {\it $L^1(\R^n)$-distance} in $\P_b(N+1).$ 

Given $\cA\in \P_b(N+1),$ 
we denote by $\hull{\cA}$ the closed convex hull of  $\bigcup\limits_{i=1}^N A_i.$ 

In view of \eqref{per_vs_anis_per} and \cite[Theorem 3.10]{BH:2018.siam} we have the following compactness result. 

\begin{proposition}[\textbf{Compactness}]\label{prop:compactness}
Let   $\cA^{(k)}\in \P_b(N+1),$ $k=1,2,\ldots,$ and $\Omega\in \K_b(\R^n)$ be such that
$$
\sup\limits_{k\ge1} \,\,\Per_\Phi(\cA^{(k)})< +\infty,\qquad \hull{\cA^{(k)}}\subseteq \Omega
\qquad \forall k\ge1. 
$$
Then there exist $\cA\in\P_b(N+1)$ and a subsequence $\{\cA^{(k_l)}\}$ converging to $\cA$ in $L^1(\R^n)$ as $l\to+\infty.$ Moreover,
$
\bigcup\limits_{i=1}^N A_j\subseteq \overline \Omega. 
$
\end{proposition}

\section{Density estimates for almost minimizers}\label{sec:density_for_almost_min}

In this section we prove density estimates for almost minimizers (see Theorem \ref{teo:density_est}). In the two-phase case without mobility, density estimates have been proven in \cite{ATW:1993,LS:1995} (see also the proof of Theorem \ref{teo:existence_GMM2f} for the case with mobility and forcing) and in the isotropic $N$-phase case  is proven in \cite{BH:2018.siam}.  The proof of Theorem \ref{teo:density_est} is similar to \cite[Theorem 3.6]{BH:2018.siam}, however, some (technical) difficulties arise when two anisotropies differ too much and this is why we need assumption 
\eqref{kappaNborder} for proving the lower-density estimates.

\begin{definition}[\textbf{Almost-minimizers}]\label{def:almost_min}
Given $\Phi=\{\phi_1,\ldots,\phi_N\},$
$\Lambda_1,\Lambda_2\ge0,$ $\alpha_1,\alpha_2>\frac{n-1}{n}$ and 
$r_0\in(0,+\infty],$ we say that a partition $\cA\in \P(N)$ is a $(\Phi,\Lambda_1,\Lambda_2,r_0,\alpha_1,\alpha_2)$-minimizer in $\R^n$ of $\Per_\Phi$ (a  $(\Lambda_1,\Lambda_2,r_0,\alpha_1,\alpha_2)$-minimizer, or also an almost-minimizer for short) if 
$$
\Per_\Phi(\cA,B_r) \le \Per_\Phi(\cB,B_{r}) +\Lambda_1 |\cA\Delta\cB|^{\alpha_1}
+\Lambda_2 |\cA\Delta\cB|^{\alpha_2}
$$
whenever $\cB\in \P(N),$ $B_r\subset\R^n$ is a ball of radius $r\in(0,r_0)$ and  $\cA\Delta\cB\strictlyincluded B_r.$  
\end{definition}

Define 
\begin{align}
\kappa_N:=&\min\limits_{1\le i<j\le N}\,\, \|\phi_i- \phi_j\|_{L^\infty(\S^{n-1})}, \label{kappaning_def} \\
\beta_1:= &\Big(\frac{c_\Phi n\omega_n^{1-\alpha_1}}
{2^{1+ \alpha_1}  \Lambda_1 }
\Big)^{\frac{1}{n\alpha_1-n+1}},\qquad  \beta_2:= \Big(\frac{c_\Phi n\omega_n^{1-\alpha_2}}
{2^{1+ \alpha_2} \Lambda_2 }
\Big)^{\frac{1}{n\alpha_2-n+1}}, \nonumber\\ 
\gamma_N:= &\frac{c_\Phi - (N-1)\kappa_N/2}{2c_\Phi+ 2(N-1)C_\Phi  - (N-1)\kappa_N}. \nonumber 
\end{align}

\begin{theorem}[\bf Density estimates for almost 
minimizers]\label{teo:density_est}
Assume that the entries of $\Phi$ satisfy \eqref{Phi_assump}. Let $\cA\in \P(N)$ be a $(\Lambda_1,\Lambda_2,r_0,
\alpha_1,\alpha_2)$-minimizer  and   $i\in\{1,\ldots,N\}.$ Then either $A_i=\emptyset$ or for any $x\in \p A_i$ and $r\in (0, \hat r_0]$   
\begin{equation}\label{upper_vol_density_a}
\dfrac{|A_i\cap B_r(x)|}{|B_r(x)|}
\le 1- \Big( \dfrac{c_\Phi}{2(c_\Phi+C_\Phi )}\Big)^n
\end{equation}
and
\begin{equation}\label{upper_per_density_a}
\dfrac{P(A_i,B_r(x))}{r^{n-1}} \le
\left(\frac{C_\Phi }{c_\Phi}+\frac12\right) \, n\omega_n,
\end{equation}
where  
\begin{equation}\label{hat_r_0}
\hat r_0:= \min\Big\{r_0,\beta_1,\beta_2\Big\}. 
\end{equation}
Moreover, if 
\begin{equation}\label{kappaNborder}
\kappa_N<\frac{2c_\Phi}{N-1}, 
\end{equation}
then for any $r\in(0,\tilde r_0]$ 
\begin{equation}\label{lower_vol_density_a}
\gamma_N^n \le \dfrac{|A_i\cap B_r(x)|}{|B_r(x)|},
\end{equation}
and
\begin{equation}\label{lower_per_density_a}
c \le \dfrac{P(A_i,B_r(x))}{r^{n-1}},
\end{equation}
where  
\begin{equation}\label{def_of_r0}
\tilde r_0:= \min\Big\{r_0,
\Big(\frac{c_\Phi}{N-1} - \frac{\kappa_N}{2}\Big)^{\frac{1}{n\alpha_1-n+1}}\beta_1,\Big(\frac{c_\Phi}{N-1} - \frac{\kappa_N}{2}\Big)^{\frac{1}{n\alpha_2-n+1}}\beta_2\Big\}  
\end{equation}
and
$$
c:=c(n,N,c_\Phi,C_\Phi ,\kappa_N):= \frac{n\omega_n (2^{1/n} -1)}{2^{1+1/n}}\,\gamma_N^{n-1}.
$$
\end{theorem}

\begin{proof} 
Without loss of generality, we assume $i=1$ and $A_i\ne\emptyset.$ 
Since $\overline{\p^*A_1}=\p A_1,$ it is sufficient to show  \eqref{upper_vol_density_a}, \eqref{upper_per_density_a}, \eqref{lower_vol_density_a}, \eqref{lower_per_density_a} when $x\in  \p^* A_1.$ For shortness, we write $B_r:=B_r(x).$ 

We start by proving  \eqref{upper_vol_density_a} and \eqref{upper_per_density_a}. Let us show
\begin{equation}\label{eeerrr}
\begin{aligned}
c_\Phi P( A_1^{(0)},B_r) \le & C_\Phi  \cH^{n-1}(A_1^{(0)}\cap \p B_r) \\
& + 2^{\alpha_1-1} \Lambda_1 |A_1^{(0)}\cap B_r|^{\alpha_1}+
2^{\alpha_2-1} \Lambda_2 |A_1^{(0)}\cap B_r|^{\alpha_2} 
\end{aligned} 
\end{equation}
for all $r \in(0,\hat r_0)$ such that 
\begin{equation}\label{good_radiusaaa}
\sum\limits_{j=1}^N \cH^{n-1}(\p B_r\cap \p^* A_j)=0. 
\end{equation}

Indeed, setting
$$
\cB:=(A_1\cup B_r,A_2\setminus B_r,\ldots,A_N\setminus B_r),
$$ 
we have $\cA\Delta\cB\strictlyincluded B_s$ for every $s\in (r,\hat r_0)$ and thus, by almost minimality, the definition \eqref{L_1dist} of  $|\cA\Delta\cB|$ and the essential disjointness of $A_j,$ 
\begin{align}\label{rytur}
0\le & \Per_\Phi(\cB,B_s) - \Per_\Phi(\cA,B_s) +
\Lambda_1|\cA\Delta\cB|^{\alpha_1} +\Lambda_2|\cA\Delta\cB|^{\alpha_2}\no \\
= & P_{\phi_1}(A_1\cup B_r,B_s) - P_{\phi_1}(A_1,B_s) 
+\sum\limits_{j=2}^N \Big(P_{\phi_j}(A_j\setminus B_r,B_s) 
- P_{\phi_j}(A_j,B_s)\Big) \\
& +2^{\alpha_1} \Lambda_1 |B_r\cap A_1^{(0)}|^{\alpha_1}+
2^{\alpha_2} \Lambda_2 |B_r\cap A_1^{(0)}|^{\alpha_2}, \no 
\end{align}
since $B_r\setminus A_1 = B_r\cap A_1^{(0)}$ up to a $\cL^n$-negligible set and 
$|\cA\Delta\cB| = 2|B_r\setminus A_1|= 2|B_r\cap A_1^{(0)}|.$
By \eqref{ess_union} and \eqref{good_radiusaaa},
\begin{equation*}
P_{\phi_1}(A_1\cup B_r,B_s) = P_{\phi_1}(A_1, B_s\setminus \overline{B_r}) +
\int_{A_1^{(0)}\cap \p B_r} \phi_1(\nu_{B_r})d\cH^{n-1},
\end{equation*}
and for any $j=2,\ldots,N,$
\begin{equation}\label{set_operation121}
P_{\phi_j}(A_j\setminus B_r,B_s) =  P_{\phi_j}(A_j,B_s\setminus \overline{B_r}) +
\int_{A_j \cap \p B_r} \phi_j(\nu_{B_r})d\cH^{n-1}.
\end{equation}
Thus  by \eqref{rytur} 
\begin{equation}\label{tanimadim_0}
\begin{aligned}
\sum\limits_{j=1}^N P_{\phi_j}(A_j,B_s) 
\le & \sum\limits_{j=1}^N P_{\phi_j}(A_j,B_s\setminus\overline{B_r}) \\
& + \int_{A_1^{(0)}\cap \p B_r} \phi_1(\nu_{B_r})d\cH^{n-1} + \sum\limits_{j=2}^N \int_{A_j\cap \p B_r} \phi_j(\nu_{B_r})d\cH^{n-1}\\
& +2^{\alpha_1} \Lambda_1 |B_r\cap A_1^{(0)}|^{\alpha_1}+
2^{\alpha_2} \Lambda_2 |B_r\cap A_1^{(0)}|^{\alpha_2}.
\end{aligned}
\end{equation}
By \eqref{Phi_assump},  the essential disjointness of $A_j$  and \eqref{good_radiusaaa} we have
\begin{align*}
&\int_{A_1^{(0)} \cap \p B_r} \phi_1(\nu_{B_r})d\cH^{n-1} + \sum\limits_{j=2}^N  \int_{A_j \cap \p B_r} \phi_j(\nu_{B_r})d\cH^{n-1}  \\
\le & C_\Phi \cH^{n-1}(A_1^{(0)}\cap \p B_r) + C_\Phi  \sum\limits_{j=2}^N \cH^{n-1}(A_j\cap \p B_r) = 2C_\Phi \cH^{n-1}(A_1^{(0)}\cap \p B_r), 
\end{align*}
thus, \eqref{tanimadim_0} and \eqref{good_radiusaaa} imply  
$$
\sum\limits_{j=1}^N P_{\phi_j}(A_j,B_r) \le 
2C_\Phi  \cH^{n-1}(A_1^{(0)} \cap \p B_r) +
2^{\alpha_1}\Lambda_1 |B_r\setminus A_1|^{\alpha_1}+
2^{\alpha_2}\Lambda_2 |B_r\setminus A_1|^{\alpha_2}.
$$ 
By \eqref{Phi_assump}, \eqref{famfor} 
and the essential disjointness of $A_j,$ 
\begin{align*}
\sum\limits_{j=2}^N P_{\phi_j}(A_j,B_r) \ge & 
c_\Phi \sum\limits_{j=2}^N P(A_j,B_r) \ge 
c_\Phi P\Big(\bigcup\limits_{j=2}^N A_j,B_r\Big)  
= c_\Phi P(A_1^{(0)},B_r), 
\end{align*}
and thus  
$
\sum\limits_{j=1}^N P_{\phi_j}(A_j,B_r) \ge 2c_\Phi P(A_1^{(0)},B_r) 
$
so that \eqref{eeerrr} follows from \eqref{hat_r_0}.

To prove \eqref{upper_vol_density_a} we add $c_\Phi \cH^{n-1}(A_1^{(0)} \cap \p B_r)$ to both sides of \eqref{eeerrr} and using $\cH^{n-1}(\p B_r\cap \p^* A_1)=0$ we get
$$
c_\Phi  P(A_1^{(0)}\cap B_r) \le (c_\Phi +C_\Phi )
\cH^{n-1}(A_1^{(0)} \cap \p B_r) + 
2^{\alpha_1-1} \Lambda_1 |A_1^{(0)}\cap B_r|^{\alpha_1}+
2^{\alpha_2-1} \Lambda_2 |A_1^{(0)}\cap B_r|^{\alpha_2}, 
$$
hence by the isoperimetric inequality 
\begin{equation}\label{diff.eq}
\begin{aligned}
c_\Phi  n\omega_{n}^{1/n}|A_1^{(0)}\cap B_r|^{\frac{n-1}{n}} \le &
(c_\Phi +C_\Phi ) \cH^{n-1}(A_1^{(0)} \cap \p B_r)\\
&+ 2^{\alpha_1-1} \Lambda_1 |A_1^{(0)}\cap B_r|^{\alpha_1} + 2^{\alpha_2-1} \Lambda_2 |A_1^{(0)}\cap B_r|^{\alpha_2}. 
\end{aligned}
\end{equation}
By the choice of $\hat r_0 $ in \eqref{hat_r_0} we have, for $l=1,2,$ 
\begin{equation}\label{gretaeee}
2^{\alpha_l-1} \Lambda_l |A_1^{(0)}\cap B_r|^{\alpha_l -\frac{n-1}{n}}\le 2^{\alpha_l-1} \Lambda_l \omega_n^{\alpha_l - \frac{n-1}{n}} \hat r_0^{n\alpha_l-n+1}  \le \frac{c_\Phi  n\omega_n^{1/n}}{4}.
\end{equation}
Inserting \eqref{gretaeee} in \eqref{diff.eq} we obtain 
$$
\dfrac{c_\Phi }{2(c_\Phi +C_\Phi )} \, n\omega_n^{1/n} 
|A_1^{(0)} \cap B_r|^{\frac{n-1}{n}} \le \cH^{n-1}(A_1^{(0)} \cap \p B_r),
$$
and whence, repeating for instance the arguments of the proof of \cite[Eq. 3.19]{BH:2018.siam}, we obtain 
$$
| A_1^{(0)} \cap B_r| \ge \Big( \dfrac{c_\Phi }{2(c_\Phi +C_\Phi )} \Big)^n  \omega_nr^n,
$$
i.e.,
$$
\dfrac{|A_1\cap B_r|}{|B_r|}
\le 1- \Big( \dfrac{c_\Phi }{2(c_\Phi +C_\Phi )} 
\Big)^n.
$$ 
From \eqref{eeerrr} and the definition of $\hat r_0$ for all $r\in(0,\hat r_0]$ we get
$$
P(A_1,B_r) \le \frac{C_\Phi }{c_\Phi } \cH^{n-1}(\p B_r) + 
\frac{2^{\alpha_1-1} \Lambda_1}{c_\Phi }
\,|B_r|^{\alpha_1} + \frac{2^{\alpha_2-1} \Lambda_2}{c_\Phi }
\,|B_r|^{\alpha_2}\le  \left(\frac{C_\Phi }{c_\Phi }+\frac12\right) \, 
n\omega_n  r^{n-1}.
$$
%

Now we prove \eqref{lower_vol_density_a} and \eqref{lower_per_density_a}. Note that assumption \eqref{kappaNborder} implies $\tilde r_0,\gamma_N>0.$ Let us show 
\begin{equation}\label{low_dens_etsasd}
\begin{aligned}
\Big(\frac{c_\Phi }{N-1} - \frac{\kappa_N}{2}\Big) P(A_1,B_r) \le &  C_\Phi  \cH^{n-1}(A_1\cap \p B_r)\\
& + 2^{\alpha_1-1} \Lambda_1 |A_1\cap B_r|^{\alpha_1}+
2^{\alpha_2-1} \Lambda_2 |A_1\cap B_r|^{\alpha_2}
\end{aligned} 
\end{equation}
for all $r\in (0,\tilde r_0)$ such that 
\begin{equation}\label{good_radius_again}
\sum\limits_{j=1}^N \cH^{n-1}(\p^*A_j\cap\p B_r) =0. 
\end{equation}
Set 
$$
I_1:=\{j\in \{2,\ldots, N\}:\,\, 
\cH^{n-1}(B_{\tilde r_0} \cap \p^* A_1\cap \p^* A_j)>0\}.
$$
Since $x\in \p A_1,$  $I_1\ne\emptyset.$ Fix $r\in(0,\tilde r_0);$ for every $j\in I_1$ consider the competitor
$$
\cB^{(j)}:=(A_1\setminus B_r,A_2,\ldots, A_{j-1},A_j\cup (A_1\cap B_r),A_{j+1},\ldots, A_N).
$$
Since $\cB^{(j)}\Delta \cA\strictlyincluded B_s$ for every $s\in (r,\tilde r_0),$ by the almost minimality of $\cA$ (recall that $\tilde r_0\le r_0$) and the equality $|\cA\Delta \cB^{(j)}|=2|A_1\cap B_r|$ one has 
\begin{equation}\label{aas3113}
\begin{aligned}
P_{\phi_1}(A_1,B_s) + P_{\phi_j}(A_j,B_s) \le 
P_{\phi_1}(A_1\setminus B_r,B_s) + P_{\phi_j}(A_j\cup(A_1\cap B_r),B_s) &\\
+2^{\alpha_1}\Lambda_1 |A_1\cap B_r|^{\alpha_1} +
2^{\alpha_2}\Lambda_2 |A_1\cap B_r|^{\alpha_2}.&
\end{aligned} 
\end{equation}
Using the equality
\begin{align*}
P_{\phi_j}(A_j\cup(A_1\cap B_r),B_s) = &
P_{\phi_j}(A_j,B_s) + P_{\phi_j}(A_1,B_r)+ 
\int_{A_1\cap \p B_r} \phi_j(\nu_{B_r})d\cH^{n-1}\\
&-  \int_{B_r\cap \p^*A_1\cap \p^* A_j} \big(\phi_j(\nu_{A_1})+
\phi_j(\nu_{A_j})\big)d\cH^{n-1}, 
\end{align*}
the analogue of \eqref{set_operation121} with $j=1$ and also \eqref{good_radius_again} in \eqref{aas3113} we establish 
\begin{align*}
P_{\phi_1}(A_1,B_s) + P_{\phi_j}(A_j,B_s) \le &
P_{\phi_j}(A_j,B_s) + P_{\phi_j}(A_1,B_r)+ 
\int_{A_1\cap \p B_r} \phi_j(\nu_{B_r})d\cH^{n-1}\\
&-  \int_{B_r\cap \p^*A_1\cap \p^* A_j} \big(\phi_j(\nu_{A_1})+
\phi_j(\nu_{A_j})\big)d\cH^{n-1} \\
&+  P_{\phi_1}(A_1,B_s\setminus B_r) +
\int_{A_1 \cap \p B_r} \phi_1(\nu_{B_r})d\cH^{n-1}\\
&+2^{\alpha_1}\Lambda_1 |A_1\cap B_r|^{\alpha_1} +
2^{\alpha_2}\Lambda_2 |A_1\cap B_r|^{\alpha_2}.
\end{align*}
Hence using $\phi_j(\nu_{A_1})=\phi_j(\nu_{A_j}),$
\begin{align*}
&2 \int_{B_r\cap \p^*A_1\cap \p^* A_j}  \phi_j(\nu_{A_1})d\cH^{n-1}  \le 
P_{\phi_j}(A_1,B_r) - P_{\phi_1}(A_1,B_r)\\
&+ 
\int_{A_1\cap \p B_r}\Big(\phi_1(\nu_{B_r})+\phi_j(\nu_{B_r})\Big)d\cH^{n-1} 
+ 2^{\alpha_1}\Lambda_1 |A_1\cap B_r|^{\alpha_1}
+ 2^{\alpha_2}\Lambda_2 |A_1\cap B_r|^{\alpha_2}. 
\end{align*}
Summing these inequalities in  $j\in I_1$ and using \eqref{Phi_assump} we get
\begin{equation} \label{rtrey} 
\begin{aligned} 
2c_\Phi \sum\limits_{j=2}^N  \cH^{n-1}(B_r\cap \p^*A_1\cap& \p^* A_j)  \le  \sum\limits_{j\in I_1} \big(P_{\phi_j}(A_1,B_r) - P_{\phi_1}(A_1,B_r)\big) \\
& +\sum\limits_{i\in I_1}\int_{A_1\cap \p B_r}\Big(\phi_1(\nu_{B_r})+\phi_j(\nu_{B_r})\Big)d\cH^{n-1}\\
& + |I_1|( 2^{\alpha_1} \Lambda_1 |A_1\cap B_r|^{\alpha_1}+
2^{\alpha_2} \Lambda_2 |A_1\cap B_r|^{\alpha_2}),
\end{aligned} 
\end{equation}
where $|I_1|$ is the number of elements of $I_1.$ 
By the definition of $I_1,$
$$
\sum\limits_{j\in I_1}  \cH^{n-1}(B_r\cap \p^*A_1\cap \p^* A_j)  = P(A_1,B_r),
$$
by the definition of $\kappa_N$ in \eqref{kappaning_def}
$$
\sum\limits_{j\in I_1} \big(P_{\phi_j}(A_1,B_r) - P_{\phi_1}(A_1,B_r)\big) \le \kappa_N |I_1|\,P(A_1,B_r),
$$
and  by \eqref{Phi_assump}
$$
\sum\limits_{i\in I_1}\int_{A_1\cap \p B_r}\Big(\phi_1(\nu_{B_r})+\phi_j(\nu_{B_r})\Big)d\cH^{n-1} \le 2C_\Phi |I_1| \cH^{n-1}(A_1\cap \p B_r).
$$
Therefore, from \eqref{rtrey} we obtain
$$
\Big(\frac{c_\Phi }{|I_1|} - \frac{\kappa_N}{2}\Big)\,P(A_1,B_r) \le C_\Phi \cH^{n-1}(A_1\cap \p B_r) + 2^{\alpha_1-1} \Lambda_1 |A_1\cap B_r|^{\alpha_1}+2^{\alpha_2-1} \Lambda_2 |A_1\cap B_r|^{\alpha_2}.
$$
Since $|I_1|\le N-1,$ inequality \eqref{low_dens_etsasd} follows.

To prove \eqref{lower_vol_density_a} we add $(\frac{c_\Phi }{N-1} - \frac{\kappa_N}{2})\cH^{n-1}(A_1\cap\p B_r)$ to both sides of \eqref{low_dens_etsasd} and get 
\begin{equation}\label{low_dens_etsasd1}
\begin{aligned}
\Big(\frac{c_\Phi }{N-1} - \frac{\kappa_N}{2}\Big) P(A_1\cap B_r) \le & \Big(\frac{c_\Phi }{N-1}  - \frac{\kappa_N}{2} + C_\Phi \Big)\cH^{n-1}(A_1\cap \p B_r)\\
& + 2^{\alpha_1-1} \Lambda_1 |A_1\cap B_r|^{\alpha_1}+
2^{\alpha_2-1} \Lambda_2 |A_1\cap B_r|^{\alpha_2},
\end{aligned}
\end{equation}
By the definition  \eqref{def_of_r0} of $\tilde r_0$  we have $r\le \tilde r_0\le \big(\frac{c_\Phi}{N-1} - \frac{\kappa_N}{2}\big)^{\frac{1}{n\alpha_l-n+1}}\beta_l$ for $l=1,2$ and therefore
\begin{align*}
2^{\alpha_l-1} \Lambda_l |A_1\cap B_r|^{\alpha_l-\frac{n-1}{n}}\le 2^{\alpha_l-1} \Lambda_l |B_{\tilde r_0}|^{\alpha_l-\frac{n-1}{n}} =\Big(\frac{c_\Phi }{N-1}-\frac{\kappa_N}{2}\Big)\,\frac{n\omega_n^{1/n}}{4},\qquad l=1,2, 
\end{align*}
and thus, by \eqref{low_dens_etsasd1} and the isoperimetric inequality,
$$
\frac{c_\Phi  - (N-1)\kappa_N/2}{2c_\Phi + 2(N-1)C_\Phi  - (N-1)\kappa_N} \, 
n\omega_n^{1/n} \,|A_1\cap B_r|^{\frac{n-1}n} \le 
\cH^{n-1}(A_1\cap \p B_r).
$$
Now integrating we get
$$
\gamma_N^n \omega_nr^n\le |A_1\cap B_r| 
$$
and \eqref{lower_vol_density_a} follows. Finally, since 
$$
\frac{c_\Phi }{2c_\Phi +2C_\Phi } >\gamma_N,
$$
from \eqref{upper_vol_density_a}, \eqref{lower_vol_density_a} and the relative isoperimetric inequality we deduce \eqref{lower_per_density_a}.
\end{proof}
 
The following volume-distance comparison appeared in a similar form also in \cite{ATW:1993,BH:2018.siam,LS:1995} and will be used in the proof of the existence  of GMM.

\begin{proposition}\label{prop:funny_estimate}
Given $\theta,r_0>0,$ let $A\in BV(\R^n;\{0,1\})$ be such that   
\begin{equation}\label{daens_per}
\theta r^{n-1} \le P(A,B_r(x)),\qquad r\in (0, r_0], 
\end{equation}
whenever $x\in \p A.$ Then for any $\ell>0$ and $B\in BV(\R^n;\{0,1\})$ one has
\begin{equation}\label{kahurfc}
|B\Delta A|\le \frac{5^n\omega_n}{\theta}\max\Big\{1,\Big(\frac{\ell}{r_0}\Big)^{n-1}\Big\} 
\, P(A)\,\ell + 
\frac{1}{\ell}\,\int_{A\Delta B} \dist(x,\p A)dx. 
\end{equation}
\end{proposition}

\begin{proof}
We follow \cite[Proposition 4.5]{BH:2018.siam} with minor modifications and we give the details for the convenience of the reader. Define
$$
E:=\{x\in B  \Delta A:\,\, \dist(x,\p A)\le \ell\},
\quad
F:=\{x\in B \Delta A:\,\, \dist(x,\p A_j)\ge \ell\}.
$$
By the Chebyshev inequality,
$$
|F| \le \frac{1}{\ell}\,\int_{F}\dist(x,\p A)dx \le 
\frac{1}{\ell}\, \int_{A\Delta B} \dist(x,\p A)dx.
$$
Let us estimate $|E|.$ By a covering argument, one can find a finite family of disjoint balls $\{B_{\ell}(x_k)\}$ $x_k\in \p A,$  such that $E$ is covered by the family $\{B_{5\ell}(x_k)\}_{k=1}^m.$
If $\ell\ge r_0,$ by \eqref{daens_per} and the disjointness of $\{B_{\ell}(x_k)\}$ (and hence of $\{B_{r_0}(x_k)\}$),  
\begin{align*}
|E| \le & \sum\limits_{k=1}^m  \omega_n (5\ell)^n =\frac{5^n\omega_n \ell^n}{\theta r_0^{n-1}}  
\sum\limits_{k=1}^m \theta r_0^{n-1}\le \frac{5^n\omega_n \ell^n}{\theta r_0^{n-1}}  
\sum\limits_{k=1}^m  P(A_j, B_{r_0}(x_k))\\
\le & 
\frac{5^n\omega_n \ell^n}{\theta  r_0^{n-1}}  
P\Big(A_j, \bigcup\limits_{k=1}^m   B_{r_0}(x_k)\Big) \le 
\frac{5^n\omega_n}{\theta }\Big(\frac{\ell}{r_0}\Big)^{n-1} \, P(A_j)\,\ell.
\end{align*} 
Analogously, if $\ell<r_0,$ then 
$$
|E| \le \frac{5^n\omega_n}{\theta }\, P(A_j)\,\ell.
$$  
Now \eqref{kahurfc} follows from the inequality $|B\Delta A|\le |E|+|F|$ and estimates for $|A|$ and $|B|.$
\end{proof}

\black 

\section{Existence of GMM for bounded partitions} \label{sec:existence_GMM_partition}

Given a norm $\psi$ in $\R^n$ and $E,F\subseteq \R^n$ set
$$
\bar\sigma_\psi(E,F):= \int_{E\Delta F} \dist_\psi(x,\p F)dx.
$$
Note that $\bar\sigma_\psi(E,F)=0$ if $|E\Delta F|=0$ whereas $\bar\sigma_\psi(E,F)=+\infty$
if $\p F =\emptyset$ and $|E\Delta F|>0.$
Moreover,  $X,Y\subseteq \R^n$ are measurable and $\p Y\ne\emptyset,$
\begin{equation*}
\begin{aligned}
\int_{X\Delta Y } \dist_\psi(x,\p Y) dx = \int_X \sdist_\psi(x,\p Y) dx - 
\int_Y\sdist_\psi(x,\p Y) dx\quad \text{if $X\cap Y$ is bounded},\\
\int_{X\Delta Y } \dist_\psi(x,\p Y) dx = \int_{Y^c} \sdist_\psi(x,\p Y) dx - 
\int_{X^c}\sdist_\psi(x,\p Y) dx\quad \text{if $X^c\cap Y^c$ is bounded}.
\end{aligned}
\end{equation*}

Given a family $\Psi:=\{\psi_1,\ldots,\psi_{N+1}\}$ of norms $\psi_i$ in $\R^n,$ and $\cA,\cB\in  \P_b(N+1),$ we set
$$
\sigma_\Psi(\cA,\cB):= \sum\limits_{i=1}^{N+1} \bar\sigma_{\psi_i}(A_i,B_i),
$$
where $N+1\ge2.$ In the literature $\Psi$ is called the set of mobilities. Since $N$ is finite, there exist $0<c_\Psi\le C_\Psi <+\infty$ such that 
\begin{equation}\label{Psi_assump}
c_\Psi \le \psi_i(\nu) \le C_\Psi,\qquad i=1,\ldots,N+1,\qquad \nu\in\S^{n-1}.  
\end{equation}
Observe that for every $\cB\in\P_b(N+1)$ the map $\sigma_\Psi(\cdot,\cB)$ is $L^1(\R^n)$-lower semicontinuous in $\P_b(N+1).$ 
\smallskip

Given families $\Phi:=\{\phi_1,\ldots,\phi_{N+1}\}$ of anisotropies and $\Psi:=\{\psi_1,\ldots,\psi_{N+1}\}$ of mobilities, and ${\bf H}:=(H_1,\ldots,H_{N+1})$ of functions $H_i\in L_\loc^1(\R^n),$ $i=1,\ldots,N,$ and $H_{N+1}\in L^1(\R^n),$ consider the functional 
$\func:\P_b(N+1)\times \P_b(N+1)\times[1,\infty)\to[-\infty,\infty],$
\begin{equation*}
\func(\cE,\cF,\lambda) = \Per_\Phi(\cE) + \force(\cE) + \lambda\sigma_\Psi(\cE,\cF),
\end{equation*}
where 
$$
\force(\cE): = \sum\limits_{i=1}^{N+1} \int_{E_i} H_idx.
$$
Note that $\func(\cdot,\cF;\lambda)$ is well-defined and $L^1(\R^n)$-lower semicontinuous in $\P_b(N+1)$.
Notice that   $\force$ can also be represented as
\begin{equation}\label{forcing}
\force(\cE) = \sum\limits_{j=1}^N \int_{E_j} (H_j- H_{N+1})dx +\int_{\R^n} H_{N+1}dx.
\end{equation}
The functional $\func$ is a generalization of the Almgren-Taylor-Wang functional \cite{ATW:1993} to the case of partitions \cite{BH:2018.siam,Degiorgi:1996} in presence of anisotropies, mobilities and external forces.

The main result of this section is the following, which generalizes \cite[Theorems 4.9 and 5.1]{BH:2018.siam} to the anisotropic case with mobilities; recall that $\kappa_{N+1}$ is defined in \eqref{kappaning_def}. 

\begin{theorem}[\bf Existence of $GMM$]\label{teo:existence_GMM}
Let $\Phi=\{\phi_1,\ldots,\phi_{N+1}\}$ and $\Psi=\{\psi_1,\ldots,\psi_{N+1}\}$ be families of anisotropies and mobilities, respectively.
Suppose that 
\begin{equation}\label{Phi_condition}
\kappa_{N+1}<\frac{2c_\Phi}{N}, 
\end{equation}
and ${\bf H}=(H_1,\ldots,H_{N+1})$ satisfies
\begin{equation}\label{H_condition}
\begin{cases}
\text{$H_i\in L_\loc^p(\R^n),$ $i=1,\ldots,N+1,$ for some $p>n$ and 
$H_{N+1}\in L^1(\R^n);$ }\\[1mm]
\text{$\exists R>0$ s.t. $H_i\ge H_{N+1}$ a.e. in $\R^n\setminus B_R(0)$ for $i=1,\ldots,N.$}
\end{cases}
\end{equation}
Then for every $\cG\in \P_b(N+1),$  $GMM(\func,\cG)$ is non empty. 
Moreover,  there exists a constant ${\rm C}=
{\rm C}(N,n,\Phi,\Psi,{\bf H},\cG)>0$ such that  for any $\cM\in GMM(\func,\cG),$ 
\begin{equation}\label{hulder_est}
|\cM(t)\Delta \cM(t')|\le  {\rm C}\,|t-t'|^{\frac{1}{n+1}},\qquad t,t'>0,\,\,|t-t'|<1
\end{equation}
and
\begin{equation}\label{dddid}
\bigcup\limits_{j=1}^N M_j(t) \subseteq\, D:=\text{closed convex hull of}\,\,\hull{\cG}\cup B_R
\qquad \forall t\ge0 
\end{equation}
and $B_R$ is not present in \eqref{dddid} if ${\bf H}\equiv0$. 
In addition, if $\sum\limits_{j=1}^{N+1} 
|\overline{G_j}\setminus G_j| = 0,$ then \eqref{hulder_est} holds for any $t,t'\ge0$ with $|t-t'|<1.$
\end{theorem}

\begin{proof}
We give only few details of the proof since it can be done following the arguments of the proofs of \cite[Theorems 4.9 and 5.1]{BH:2018.siam}.  

{\it Step 1: Existence of minimizers.} Given $\cA\in \P_b(N+1)$ and $\lambda\ge1,$ the problem 
\begin{equation*} 
\inf\limits_{\cB\in \P_b(N+1)} \func(\cB,\cA;\lambda) 
\end{equation*}
has a solution. Moreover, every minimizer $\cA(\lambda)=(A_1(\lambda),\ldots, A_{N+1}(\lambda))$ 
satisfies the bound 
\begin{equation*}
\bigcup\limits_{i=1}^N A_i(\lambda)\subseteq\,\, \text{closed convex hull of} \,\,\hull{\cA}\cup B_R(0).
\end{equation*}
We omit the proof since it is proven along the same lines as \cite[Theorem 4.2]{BH:2018.siam} using the anisotropic  Comparison Theorem with convex sets\footnote{If $E\in BV(\R^n;\{0,1\}),$ then $ P_\phi(E) \ge P_\phi(E\cap C)$ for every anisotropy $\phi$ and every closed convex set  $C\subset\R^n.$  }   and the inequality $\dist_\psi(\cdot,E_0)>0$ in any $F\subset \R^n\setminus \overline{E_0}.$  

{\it Step 2: Density estimates for minimizers.} Let $\cA\in \P_b(N+1)$ satisfy $\hull{\cA}\subset D$ and  set 
\begin{equation}\label{lambda1_lambda2}
\Lambda_1:=\lambda\max\limits_{1\le j\le N+1} \, (\diam_{\psi_j} D +2),\qquad 
\Lambda_2:= N^{1/p}\max\limits_{1\le j\le N} \|H_j-H_{N+1}\|_{L^p(D_1)}, 
\end{equation}
where $D_1:=\{x\in\R^n:\,\, \dist(x,D)\le1\}.$
Let $\lambda \ge1$ and $\cA(\lambda)\in \P_b(N+1)$  be a minimizer of $\func(\cdot, \cA;\lambda).$  Then for every $i\in \{1,\ldots,N+1\}$  either $\p A_i(\lambda)$ is empty or  for any $x\in \p A_i(\lambda)$ and 
\begin{equation}\label{youtube_tarmogi}
r\in \Big(0, \min\Big\{1,\Big(\frac{c_\Phi}{N} - \frac{\kappa_{N+1}}{2}\Big)\,
\frac{n}{4\Lambda_1}, 
\Big[\Big(\frac{c_\Phi}{N} - \frac{\kappa_{N+1}}{2}\Big)\,\frac{n\omega_n^{1/p}}{2^{2-1/p}\Lambda_2}
\Big]^{\frac{p}{p-n}}\Big\}\Big]
\end{equation}
one has
\begin{equation}\label{eq:vol.density_est}
\Big(\frac{2c_\Phi - N\kappa_{N+1}}{2c_\Phi +2NC_\Phi- N\kappa_{N+1}}\Big)^n 
\le \dfrac{|A_i(\lambda)\cap B_r(x)|}{|B_r(x)|}
\le 1- \Big( \dfrac{c_\Phi}{2(c_\Phi+C_\Phi)}\Big)^n
\end{equation}
and 
\begin{equation}\label{eq:per.density_est}
c^\Phi \le \frac{P(A_i(\lambda),B_r(x))}{r^{n-1}} 
\le \left(\frac{C_\Phi}{c_\Phi}+\frac12\right) \, n\omega_n,
\end{equation}
where $\kappa_{N+1}$ is given by \eqref{kappaning_def} and
$$
c^\Phi = c^\Phi(N,n):= \frac{n\omega_n(2^{1/n} -1)}{2^{1+1/n}}\,\left(\frac{2c_\Phi - N\kappa_{N+1}}{2c_\Phi+ 2NC_\Phi - N\kappa_{N+1}} \right)^{n-1}.
$$
The proof is analogous to the the proof of \cite[Theorem 4.6]{BH:2018.siam}: we only show that
$$
\text{$\cA(\lambda)$ is a $(\Phi,\Lambda_1,\Lambda_2, 1, 1-1/p)$-minimizer,}
$$
and hence \eqref{eq:vol.density_est}-\eqref{eq:per.density_est} follow from Theorem \ref{teo:density_est}. 
Let $\cC\in\P_b(N+1)$ be such that $\cC\Delta \cA(\lambda) \strictlyincluded B_\rho(x)$ with $\rho\in (0,1).$ By the minimality of $\cA(\lambda),$ 
\begin{align*}
\Per_\Phi(\cA(\lambda),B_\rho(x)) \le  & \Per_\Phi(\cC,B_\rho(x)) \\
&+\lambda\sum\limits_{j=1}^{N+1}  
\int_{C_j\Delta A_j(\lambda)}
\dist_{\psi_j}(x,\p A_j) dx
+ \sum\limits_{j=1}^{N}  
\int_{C_j\Delta A_j(\lambda)}|H_j-H_{N+1}|dx. 
\end{align*}
By Step 1, $\hull{\cA(\lambda)}\subseteq D,$  thus 
$$
\dist_{\psi_j}(z,\p A_j) \le \diam_{\psi_j} D +2\rho\qquad \text{for all $j=1,\ldots,N+1$ and $z\in \cC\Delta \cA(\lambda),$}
$$
where  $\diam_{\psi_j}$ is the $\psi_j$-diameter of a set.
Then, since $\cC\Delta\cA(\lambda)\subset D_1,$
\begin{align*}
\sum\limits_{j=1}^{N+1}  
\int_{C_j\Delta A_j(\lambda)}
\dist_{\psi_j}(x,\p A_j)\,dx \le &\max\limits_{1\le j\le N+1} \,\,(\diam_{\psi_j} D +2)\,|\cC\Delta \cA(\lambda)|  
\end{align*}
and
\begin{align*}
\sum\limits_{j=1}^{N}  
\int_{C_j\Delta A_j(\lambda)}|H_j-H_{N+1}|dx\le& 
\sum\limits_{j=1}^{N}  
|C_j\Delta A_j(\lambda)|^{1-1/p} \|H_j-H_{N+1}\|_{L^p(D_1)}
\\
\le &
N^{1/p}\max\limits_{1\le j\le N} \|H_j-H_{N+1}\|_{L^p(D_1)}\,
|\cC\Delta \cA(\lambda)|^{1-1/p}.
\end{align*}
Thus,
$$
\Per_\Phi(\cA(\lambda),B_\rho(x)) \le \Per_\Phi(\cC,B_\rho(x)) + 
\Lambda_1 |\cC\Delta\cA(\lambda)|+
\Lambda_2 |\cC\Delta\cA(\lambda)|^{1-1/p}.
$$
  
{\it Step 3: Existence of GMM.}  
Given $\lambda\ge1$ and $k\in \N_0$ we  define $\cG(\lambda,k)$ recursively as: $\cG(\lambda,0)=\cG$ and
$$
\func(\cG(\lambda,k),\cG(\lambda,k-1),\lambda)
=\min\limits_{\cA\in \P_b(N+1)} \func(\cA,\cG(\lambda,k-1),\lambda).
$$
Since $\func(\cG(\lambda,k),\cG(\lambda,k-1),\lambda) \le \func(\cG(\lambda,k-1),\cG(\lambda,k-1),\lambda),$ 
we have  
\begin{align}\label{dar_monoton}
\Per_\Phi( \cG(\lambda,k) )+\force( \cG(\lambda,k)) + &
\lambda \sigma_\Psi(\cG(\lambda,k),\cG(\lambda,k-1))\no \\ 
&\le   \Per_\Phi(\cG(\lambda,k-1))+ \force( \cG(\lambda,k-1)) .
\end{align}
Thus, the map  
$
k\in\N_0\mapsto \Per_\Phi(\cG(\lambda,k)) + \force( \cG(\lambda,k))
$
is non-increasing for any $\lambda\ge1.$ In particular,
\begin{align}\label{per_estimatesaa}
\Per_\Phi(\cG(\lambda,k))\le &\Per_\Phi(\cG(\lambda,0)) +
\sum\limits_{j=1}^N \int_{G_j(\lambda,k)\Delta G_j(\lambda,0)} 
|H_j-H_{N+1}|dx \no \\
\le & \Per_\Phi(\cG) +
\sum\limits_{j=1}^N \|H_j-H_{N+1}\|_{L^1(D)}=:\mu_0,\qquad k\ge0  
\end{align}
and
\begin{equation}\label{dddid1}
\bigcup\limits_{j=1}^N G_j(\lambda,k) \subseteq\, D
\qquad \text{for all $\lambda\ge1,$ and $k\ge0.$}
\end{equation}

Fix $t,t'>0$ with $0<t-t'<1.$ Let $\lambda>1$ be so large (depending on $t,$ $t',$ $n,$ $N,$ ${\bf H},$ $\Psi$ and $\Phi$) that  setting $k_0=[\lambda t'],$ $m_0=[\lambda t],$  one has $m_0 \ge k_0 +3\ge4$ and  
$$
r_\lambda:=\min\Big\{1,\Big(\frac{c_\Phi}{N} - \frac{\kappa_{N+1}}{2}\Big)\,
\frac{n}{4\Lambda_1}, 
\Big[\Big(\frac{c_\Phi}{N} - \frac{\kappa_{N+1}}{2}\Big)\,\frac{n\omega_n^{1/p}}{2^{2-1/p}\Lambda_2}
\Big]^{\frac{p}{p-n}}\Big\}=\frac{\gamma}{\lambda},
$$
where $\Lambda_1$ and $\Lambda_2$ are given in \eqref{lambda1_lambda2}, and recalling \eqref{Phi_condition},
$$
\gamma:=\Big(\frac{c_\Phi}{N} - \frac{\kappa_{N+1}}{2}\Big)\,
\frac{n}{4\max\limits_{1\le j\le N+1} \, (\diam_{\psi_j} D +2)}>0.
$$
By \eqref{eq:per.density_est} for such $\lambda$ and for any $k\ge1$ any minimizer $\cG(\lambda,k)$ satisfies  
$$
P(G_j(\lambda,k),B_r(x))\ge c^\Phi r^{n-1}
$$
for any $x\in \p G_j(\lambda,k)$ and $r\in(0,r_\lambda)$ provided $\p G_j(\lambda,k)$ is non-empty. Therefore, by Proposition \ref{prop:funny_estimate} applied with $r_0=r_\lambda,$  $\theta=c^\Phi,$ $A=G_j(\lambda,k-1),$   $B=G_j(\lambda,k) $ and $\ell=r_\lambda|t-t'|^{-\frac{1}{n+1}}>r_\lambda,$ for any $j\in\{1,\ldots,N+1\}$ and  $k\in\{k_0+1,\ldots,m_0\},$ we have
$$
\begin{aligned}
|\cG(\lambda,k-1)\Delta\cG(\lambda,k)| \le & \frac{5^n\omega_n \gamma }{\lambda c^\Phi |t-t'|^{\frac{n}{n+1}}} \Per(\cG(\lambda,k-1))\\
& +  \frac{\lambda |t-t'|^{\frac{1}{n+1}}}{\gamma}\,\sigma(\cG(\lambda,k),\cG(\lambda,k-1)).
\end{aligned}
$$
Now the bounds \eqref{Phi_assump}, \eqref{Psi_assump} and  \eqref{dar_monoton} imply 
$$
\begin{aligned}
&|\cG(\lambda,k-1)\Delta\cG(\lambda,k)| \le   \frac{5^n\omega_n \gamma}{\lambda c_\Phi c^\Phi  |t-t'|^{\frac{n}{n+1}}} \Per_\Phi(\cG(\lambda,k-1))\\
& +  \frac{|t-t'|^{\frac{1}{n+1}}}{\gamma c_\Psi}\,\Big(\Per_\Phi(\cG(\lambda,k-1)) + \force(\cG(\lambda,k-1))- \Per_\Phi(\cG(\lambda,k) -\force(\cG(\lambda,k))\Big).
\end{aligned}
$$
Summing this inequality in $k\in\{k_0+1,\ldots,m_0\},$ we obtain
\begin{align}\label{ppppoiu}
&|\cG(\lambda,[\lambda t])\Delta \cG(\lambda,[\lambda t'])|\le 
\sum\limits_{k=k_0+1}^{m_0} 
|\cG(\lambda,k-1)\Delta \cG(\lambda,k)| \no \\
\le  &\frac{5^n\omega_n\gamma }{\lambda c_\Phi c^\Phi  |t-t'|^{\frac{n}{n+1}}} 
\,\sum\limits_{k=k_0+1}^{m_0} \Per_\Phi(\cG(\lambda,k-1)) \\
&+\frac{|t-t'|^{\frac{1}{n+1}}}{\gamma c_\Psi }\,\Big(\Per_\Phi(\cG(\lambda,k_0)) + \force(\cG(\lambda,k_0))- \Per_\Phi(\cG(\lambda,m_0) -\force(\cG(\lambda,m_0))\Big).\no
\end{align}
By \eqref{per_estimatesaa}
$$
\sum\limits_{k=k_0+1}^{m_0} \Per_\Phi(\cG(\lambda,k-1)) \le \mu_0(m_0-k_0)\le \lambda \mu_0\Big(|t-t'| +\frac1\lambda\Big) 
$$
and 
$$
\begin{aligned}
\Per_\Phi(\cG(\lambda,k_0))+\force(\cG(\lambda,k_0))  &- \Per_\Phi(\cG(\lambda,m_0))-\force(\cG(\lambda,m_0))\\
& \le \Per_\Phi(\cG(\lambda,k_0))+\sum\limits_{j=1}^N\|H_j-H_{N+1}\|_{L^1(D)} \le 2\mu_0.
\end{aligned}
$$
Thus, from \eqref{ppppoiu} we get 
\begin{equation}\label{eq:muhim_holder}
|\cG(\lambda,[\lambda t])\Delta \cG(\lambda,[\lambda t'])|\le
{\rm C} |t-t'|^{\frac{1}{n+1}} + \widetilde{\rm C} |t-t'|^{-\frac{n}{n+1}}
\lambda^{-1}, 
\end{equation}
where 
$$
{\rm C}:= \frac{5^n\omega_n \gamma  \mu_0}{c_\Phi c^\Phi} + \frac{2\mu_0}{\gamma c_\Psi }  \quad 
 {\rm and}\quad \widetilde{\rm C}:= \frac{5^n\omega_n \gamma  \mu_0}{c_\Phi c^\Phi}. 
$$
The remaining part of the proof is as the proofs of \cite[Theorems 4.9 and 5.1]{BH:2018.siam}. 
We note here that if $\cG\in \P_b(N+1)$ satisfies $\sum\limits_{j=1}^{N+1}|\overline{G_j}\setminus G_j|=0,$ then $\frac{1}{n+1}$-H\"olderianity of GMM at time $t=0$ follows from the relations 
$$
\lim\limits_{\lambda\to+\infty} |\cG(\lambda,1)\Delta \cG|=0,\qquad \lim\limits_{\lambda\to+\infty} \Per_\Phi(\cG(\lambda,1))=\Per_\Phi(\cG),
$$
$$
\lim\limits_{\lambda\to+\infty} \lambda\sigma_\Psi(\cG(\lambda,1), \cG)=0,\qquad \lim\limits_{\lambda\to+\infty} \force(\cG(\lambda,1))=\force(\cG) 
$$
whose proofs can be done following the arguments of \cite[Proposition 4.5]{BH:2018.siam}.
\end{proof}

\subsection{Two-phase case}
 
When $N=1,$ repeating the arguments of \cite{LS:1995} in our more general setting, we can improve the H\"older exponent of GMM to $1/2$ without any restriction on the anisotropies.  

\begin{theorem}
\label{teo:existence_GMM2f}
Let $\Phi=(\phi_1,\phi_2)$ and $\Psi=(\psi_1,\psi_2),$ and assume that ${\bf H}=(H_1,H_2)$ satisfies \eqref{H_condition} with $N=1.$
Then for every $\cG\in \P_b(2),$  $GMM(\func,\cG)$ is non empty. Moreover, there exists a constant ${\rm C}={\rm C}(n,\Phi,\Psi,{\bf H},\Per_\Phi(\cG))>0$ such that  for any $\cN\in GMM(\func,\cG),$ 
\begin{equation}\label{hulder_est2} 
|\cN(t)\Delta \cN(t')|\le  {\rm C}\,|t-t'|^{1/2},\qquad t,t'>0,\,\,|t-t'|<1
\end{equation}
and
\begin{equation*} 
 N_1(t) \subseteq\, \text{closed convex hull of}\,\,G_1\cup B_R
\qquad \forall t\ge0.
\end{equation*}
In addition, if $|\overline{G_1}\setminus G_1| = 0,$ then \eqref{hulder_est2} holds for any $t,t'\ge0$ with $|t-t'|<1.$
\end{theorem}

The proof runs along the same lines of Theorem \ref{teo:existence_GMM} with however an improved bound for the radii in the proof of the density estimates, see \eqref{good_bound_radisu}  below.  We need to make a detailed proof since this will be used in the proof of Theorem \ref{teo:existence_GMM_half}.

\begin{proof}
Letting $\phi:=\phi_1+\phi_2,$  $H:=H_1-H_2,$ and 
$$
d_\psi^E(\cdot):=\dist_{\psi_1}(\cdot,\p E) +\dist_{\psi_2}(\cdot,\p E),
\qquad 
\tilde d_\psi^E(\cdot):=\sdist_{\psi_1}(\cdot,\p E) +\sdist_{\psi_2}(\cdot,\p E),
$$  
we have 
\begin{equation}\label{funktwoe}
\begin{aligned}
& \func(\cA,\cB,\lambda) -\int_{\R^n} H_2dx =  P_\phi(A_1)  + \int_{A_1} Hdx + \lambda \int_{A_1\Delta B_1} d_\psi^{B_1} dx \\
 =  & P_\phi(A_1) + \int_{A_1} Hdx + \lambda \int_{A_1} \tilde d_\psi^{B_1} dx - \lambda \int_{B_1} \tilde d_\psi^{B_1} dx=:\functwo(A_1,B_1,\lambda).
\end{aligned}
\end{equation}
Therefore, it suffices to show that for any bounded $G\in BV(\R^n;\{0,1\}),$ $GMM(\functwo,G)$ is nonempty and there exists $C_0:=C_0(n,\Phi,\Psi,H,P_\phi(G))$  such that for any $L\in GMM(\functwo,G)$
\begin{equation}\label{sokmang_otkanlar}
|L(t)\Delta L(s)| \le C_0\,|t-t'|^{1/2},\qquad t,t'>0,\qquad |t-t'|<1 
\end{equation}
and 
\begin{equation}\label{rasvoi_dajal}
L(t) \subseteq\, D:=\hull{G\cup B_R},\qquad t\ge0. 
\end{equation}

Note that, except for the presence of $\int_{A_1} Hdx,$ $\functwo$ is of the form of the Almgren-Taylor-Wang functional. 

We divide the proof into five steps. 
\smallskip 

{\it Step 1:  Existence of minimizers.} Let $E_0\in BV(\R^n;\{0,1\})$ be such that $E_0\subset D.$  Since $\phi$ is a norm, $\tilde d_\psi^{E_0}\ge 0$ in $\R^n\setminus E_0$ and $H\ge0$ in $\R^n\setminus D,$ as in the Euclidean two-phase case (see e.g. \cite{Ambrosio:1995}) we can use the Comparison Theorem with the convex set $D$ to establish the existence of a minimizer of $\functwo(\cdot,E_0,\lambda)$ and also that every minimizer $E_\lambda$ satisfies $E_\lambda \subseteq D.$ 
\smallskip

{\it Step 2: Unconstrained density estimates for minimizers.} Let $E_\lambda$ minimize  $\functwo(\cdot,E_0,\lambda)$ and $x_0\in E_\lambda\Delta E_0$ be such that $\dist(x_0,\p E_0)\ge  r_1$ for some $r_1>0$ satisfying 
\begin{equation}\label{r_1_def}
w_n^{1/n-1/p} r_1^{1-n/p}\Big(\int_D |H|^pdx\Big)^{1/p}< c_\Phi n\omega_n^{1/n}. 
\end{equation}
Notice that there are no restrictions on $r_1>0;$ in addition$x_0$ needs not be on $\p E_\lambda.$ Let us show that
\begin{itemize}
\item[(a)] if $x_0\in E_\lambda\setminus E_0,$ then 
$$
\frac{|E_\lambda\cap B_r(x_0)|}{|B_r(x_0)|} \ge \Big(\frac{c_\Phi}{4C_\Phi}\Big)^n \qquad \text{for any $r\in(0,r_1)$};
$$

\item[(b)] if $x_0\in E_0\setminus E_\lambda,$ then 
$$
\frac{|B_r(x_0)\setminus E_\lambda|}{|B_r(x_0)|} \ge \Big(\frac{c_\Phi}{4C_\Phi}\Big)^n \qquad \text{for any $r\in(0,r_1)$}.
$$
\end{itemize}

We prove only (a), since the proof of  (b) is similar. For shortness we write $B_r:=B_r(x_0).$ Fix any $r\in(0,r_1)$ such that 
\begin{equation}\label{good_choice_rad}
\cH^{n-1}(\p^* E_\lambda\cap \p B_r)=0. 
\end{equation}
%
By the minimality of $E_\lambda$ we have $\functwo(E_\lambda,E_0,\lambda)\le \functwo(E_\lambda\setminus B_r,E_0,\lambda)$ so that 
\begin{equation}\label{balabaji}
P_\phi(E_\lambda,B_s) + \lambda\int_{E_\lambda\cap B_r} \tilde d_\psi^{E_0} dx \le P_\phi(E_\lambda\setminus B_r,B_s) + \int_{E_\lambda\cap B_r} |H|dx  
\end{equation}
for any $s>r.$  The choice of $x_0$ and the definition of $\tilde d_\psi^{E_0}$ imply $\tilde d_\psi^{E_0}\ge0 $ in $B_{r_1}$ and hence using \eqref{good_choice_rad}, \eqref{set_operation121} (applied with $\phi_j=\phi$ and $A_j=E_\lambda$) and the inclusion $(E_\lambda\setminus B_r)\Delta E_\lambda\subset B_r$, from \eqref{balabaji} we get 
\begin{equation}\label{before_isop_ineq}
P_\phi(E_\lambda\cap B_r) \le 2\int_{E_\lambda\cap \p B_r} \phi(\nu_{B_r}) d\cH^{n-1} +\int_{E_\lambda\cap B_r} |H|dx. 
\end{equation}
The definition of $\phi,$ \eqref{Phi_assump}, the isoperimetric inequality and the H\"older inequality yield
\begin{equation}\label{firts_aserr}
\begin{aligned}
2c_\Phi n\omega_n^{1/n} |E_\lambda\cap B_r|^{\frac{n-1}{n}}\le & 4C_\Phi   \cH^{n-1}(E_\lambda\cap \p B_r)\\
&+ |E_\lambda\cap B_r|^{1-\frac{1}{p}} \Big(\int_{E_\lambda\cap B_r} |H|^pdx\Big)^{1/p}.
\end{aligned}
\end{equation}
Recall by Step 1 that $E_\lambda\subseteq D.$ Thus from the inequality 
$$
|E_\lambda\cap B_r|^{1-\frac{1}{p}}  \le |E_\lambda\cap B_r|^{\frac{n-1}{n}}\, |B_r|^{\frac{1}{n}-\frac{1}{p}} =\omega_n^{1/n-1/p} r^{1-n/p} |E_\lambda\cap B_r|^{\frac{n-1}{n}} 
$$ 
and \eqref{r_1_def}, 
$$
|E_\lambda\cap B_r|^{1-\frac{1}{p}} \Big(\int_{E_\lambda\cap B_r} |H|^pdx\Big)^{1/p} \le c_\Phi n\omega_n^{1/n} |E_\lambda\cap B_r|^{\frac{n-1}{n}},
$$
and therefore, from \eqref{firts_aserr} we deduce 
$$
c_\Phi n\omega_n^{1/n} |E_\lambda\cap B_r|^{\frac{n-1}{n}}\le 4C_\Phi  \cH^{n-1}(E_\lambda\cap \p B_r).
$$
Now integrating we get 
$$
\frac{|E_\lambda\cap B_r|}{|B_r|} \ge \Big(\frac{c_\Phi}{4C_\Phi}\Big)^n 
$$
for any $r\in(0,r_1).$
\smallskip

The next step is valid in the two-phase case. We miss the proof of a
similar statement in the multiphase case  because we are not able to prove 
the analogue of Step 2\footnote{In the multiphase case we miss the analogue
of \eqref{before_isop_ineq}, that was obtained neglecting 
the term $\int_{E_\lambda \cap B_r} \widetilde d_\psi^{E_0}~dx$ in 
\eqref{balabaji}. 
For instance, 
in the planar $4$-phase, at a triple junction
involving $\phi_1,\phi_2,\phi_3$ and surrounded by the fourth
phase having $\phi_4$ as surface tension,
it is conceivable that, if $\phi_1,\phi_2,\phi_3$ are 
quite large compared to $\phi_4$, 
then around the triple point, the fourth phase appears
after one minimization step.}.
 
We essentially follow the arguments of \cite{LS:1995,MSS:2016}. 
Let 
$$
C_1=C_1(n,\Phi,\Psi)=8C_\Psi\,\Big( \frac{(4C_\Phi)^{n+1} n}{2c_\Psi c_\Phi^n}\Big)^{1/2} 
$$ 
and 
$$
C_2=C_2(n,\Phi,\Psi,H,p):=(nc_\Phi)^{\frac{2p}{n-p}}\Big(\frac{C_1}{2C_\Psi}\Big)^2 \Big(\frac1{\omega_n}\int_D |H|^pdx\Big)^{ \frac{2}{p-n} }.
$$

{\it Step 3: $L^\infty$-bound for minimizers.} For any $\lambda>C_2,$ 
if $E_\lambda$ minimizes $\functwo(\cdot,E_0,\lambda)$ then
$$
\sup\limits_{x\in E_\lambda\Delta E_0} \,\,d_\psi^{E_0}(x) \le C_1\lambda^{-1/2}.
$$ 

Assume by contradiction that there exists $x_0\in E_\lambda\Delta E_0 $ such that $d_\psi^{E_0}(x_0)>C_1\lambda^{-1/2}.$  Then from \eqref{Psi_assump} we get $\dist(x_0,\p E_0)>\frac{C_1}{2C_\Psi}\,\lambda^{-1/2}.$ Since $\lambda>C_2,$ 
we can choose $\epsilon>0$ such that $r_1:=2r=\dist(x_0,\p E_0)>(\frac{C_1}{2C_\Psi} + \epsilon)\lambda^{-1/2}$ satisfies \eqref{r_1_def}, where for shortness we drop the dependence of $r$ on $n,$ $\lambda,$ $\epsilon,$ $\Phi$ and $\Psi$. Setting $B_{r}:=B_{r}(x_0)$ without loss of generality we also suppose that \eqref{good_choice_rad} holds. 
First we assume $x_0\in E_\lambda\setminus E_0.$ Then the minimality of $E_\lambda$ implies $\functwo(E_\lambda,E_0,\lambda) \le \functwo(E_\lambda\setminus B_{r},E_0,\lambda)$ so that, similarly to \eqref{before_isop_ineq}, 
\begin{equation}\label{izzat_nasib_etgin}
P_\phi(E_\lambda\cap B_{r}) + \lambda\int_{E_\lambda\cap B_{r}} \tilde d_\psi^{E_0} dx \le 2\int_{E_\lambda\cap\p B_{r}} \phi(\nu_{B_{r}})d\cH^{n-1} + \int_{E_\lambda\cap B_{r}} |H|dx. 
\end{equation}
By the H\"older inequality,  the inclusion $E_\lambda\subset D$ and \eqref{r_1_def},
\begin{equation}\label{bemajol}
\begin{aligned}
\int_{E_\lambda\cap B_{r}} |H|dx \le & |E_\lambda\cap B_{r}|^{1-1/p} \Big(\int_{E_\lambda\cap B_{r}} |H|^pdx\Big)^{1/p}\\
\le & \omega_n^{1/n-1/p}r^{1-n/p}\, \Big(\int_{D} |H|^pdx\Big)^{1/p}\, |E_\lambda\cap B_{r}|^{\frac{n-1}{n}}, \\
< & c_\Phi n\omega_n^{1/n} |E_\lambda\cap B_{r}|^{\frac{n-1}{n}}, 
\end{aligned} 
\end{equation}
therefore, by \eqref{Phi_assump} and the isoperimetric inequality, 
$$
P_\phi(E_\lambda\cap B_{r}) > \int_{E_\lambda\cap B_{r}} |H|dx.
$$
This and \eqref{izzat_nasib_etgin} imply 
\begin{equation}\label{ahli_jahondin_yaxshilik}
\lambda\int_{E_\lambda\cap B_{r}} d_\psi^{E_0} dx <  2\int_{E_\lambda\cap\p B_{r}} \phi(\nu_{B_{r}})d\cH^{n-1}. 
\end{equation}
By \eqref{Psi_assump}, the choice of $x_0$ and the definition of $r$  one has $d_\psi^{E_0}\ge c_\Psi\dist(\cdot,\p E_0)\ge 2c_\Psi r$ in $B_{r}.$ Thus,  from \eqref{ahli_jahondin_yaxshilik} and \eqref{Phi_assump}  we get 
\begin{equation*}
2c_\Psi \lambda r |E_\lambda\cap B_{r}| < 4C_\Phi\cH^{n-1}(E_\lambda\cap\p B_{r}). 
\end{equation*}
This, the inequality $\cH^{n-1}(E_\lambda\cap\p B_{r})\le n\omega_n r^{n-1}$ and Step 2 (a) imply 
\begin{equation*}
2c_\Psi \lambda \omega_n r^{n+1}\Big(\frac{c_\Phi}{4C_\Phi}\Big)^n < 4 C_\Phi n\omega_nr^{n-1}. 
\end{equation*}
Therefore, by the definition of $C_1$ and $r,$
$$
\Big( \frac{C_1}{8C_\Psi}\Big)^2 = \frac{(4C_\Phi)^{n+1} n}{2c_\Psi c_\Phi^n} > \lambda r^2  > \Big(\frac{C_1}{8C_\Psi} + \frac{\epsilon}{4}\Big)^2,
$$
a contradiction.  

If $x_0\in E_0\setminus E_\lambda,$ then we use $\functwo(E_\lambda,E_0,\lambda) \le \functwo(E_\lambda\cup B_{r},E_0,\lambda)$  and repeat the similar arguments above.
\smallskip 

Before passing to the next step let us define
$$
C_3:=C_3(n,\Phi,\Psi) = \frac{2nc_\Phi}{C_1+ \sqrt{C_1^2 + 4nc_\Phi C_\Psi}} ,
$$
$$
C_4:= C_4(n,\Phi)=\frac{n\omega_n (2^{1/n} - 1)}{2^{n+1/n}}\,\Big(\frac{c_\Phi}{4C_\Phi}\Big)^{n-1}
$$
and 
$$
C_5=C_5(n,\Phi,\Psi,H,p):=\max\Big\{C_2,C_3^2\Big(\frac{nc_\Phi}{2}\Big)^{\frac{2p}{n-p}}\Big(\frac1{\omega_n}\int_D |H|^pdx\Big)^{ \frac{2}{p-n} }\Big\}.
$$

{\it Step 4: Uniform density estimates for minimizers.} Given $\lambda>C_5$ and a minimizer $E_\lambda$ of $\functwo(\cdot,E_0,\lambda),$ following arguments of \cite{LS:1995,MSS:2016}  let us show  that 
\begin{equation}\label{volume_dens_two}
\Big(\frac{c_\Phi}{4C_\Phi}\Big)^n \le \frac{|E_\lambda\cap B_r(x)|}{|B_r(x)|} \le 1-  \Big(\frac{c_\Phi}{4C_\Phi}\Big)^n 
\end{equation}
and 
\begin{equation}\label{perimeter_dens_two}
C_4 \le \frac{P(E_\lambda, B_r(x))}{r^{n-1}} \le \frac{2C_\Phi + c_\Phi}{2c_\Phi} n\omega_n  
\end{equation}
for any $x\in \p E_\lambda$  and 
\begin{equation}\label{good_bound_radisu}
r\in (0,C_3\lambda^{-1/2}).
\end{equation}

Since $E_\lambda^{(1)}=E_\lambda$ and  $\overline{\p^*E_\lambda} =\p E_\lambda,$ we can suppose $x\in \p^*E_\lambda.$ 
For any  $r$ as in \eqref{good_bound_radisu} and $y\in B_r(x)$ one has 
$$
d_\psi^{E_0}(y)\le d_\psi^{E_\lambda}(y) + \sup\limits_{z\in E_\lambda\Delta E_0} d_\psi^{E_0}(z) \le 2C_\Psi r + \sup\limits_{z\in E_\lambda\Delta E_0} d_\psi^{E_0}(z) 
$$
so that by Step 3   
\begin{equation}\label{l_cheksiz_est_dpsi}
d_\psi^{E_0}(y) \le (2C_\Psi C_3 + C_1)\lambda^{-1/2}.  
\end{equation}

Let us prove the lower volume density estimates. For shortness set $B_r:=B_r(x).$ Let $r\in(0,C_3\lambda^{-1/2})$ be such that \eqref{good_choice_rad} holds. As in the proof of Step 2, from the inequality $\functwo(E_\lambda,E_0,\lambda) \le \functwo(E_\lambda\setminus B_r,E_0,\lambda)$  we get 
\begin{equation}\label{per_estimate}
P_\phi(E_\lambda, B_r) \le \int_{E_\lambda\cap\p B_r} \phi(\nu_{B_r})d\cH^{n-1} + \lambda\int_{E_\lambda\cap B_r} d_\psi^{E_0}(y)dy + \int_{E_\lambda\cap B_r} |H|dy. 
\end{equation}
By \eqref{l_cheksiz_est_dpsi}, the choice of  $r$ and the equality 
$$
(2C_\Psi C_3 + C_1)C_3 = \frac{nc_\Phi}{2}, 
$$
we have
\begin{align*}
\lambda \int_{E_\lambda\cap B_r} d_\psi^{E_0}(y)dy \le & (2C_\Psi C_3 + C_1) \omega_n^{1/n}\lambda^{1/2}r |E_\lambda\cap B_r|^{\frac{n-1}{n}}  \\
\le &  \omega_n^{1/n}(2C_\Psi C_3 + C_1)C_3 |E_\lambda\cap B_r|^{\frac{n-1}{n}} = \frac{c_\Phi n\omega_n^{1/n}}{2}\,|E_\lambda\cap B_r|^{\frac{n-1}{n}}.
\end{align*}
Furthermore, using $\lambda>C_5,$ as in \eqref{bemajol} 
\begin{align*}
\int_{E_\lambda\cap B_r} |H|dx \le &
 \omega_n^{1/n-1/p}r^{1-n/p}\, \Big(\int_{E_\lambda\cap B_{r_1}} |H|^pdx\Big)^{1/p}\, |E_\lambda\cap B_r|^{\frac{n-1}{n}} \\
\le &   \omega_n^{1/n-1/p}(C_3\lambda^{-1/2})^{1-n/p}\, \Big(\int_{E_\lambda\cap B_{r_1}} |H|^pdx\Big)^{1/p}\, |E_\lambda\cap B_r|^{\frac{n-1}{n}}\\
\le & \frac{c_\Phi n\omega_n^{1/n}}{2}\,|E_\lambda\cap B_r|^{\frac{n-1}{n}}.
\end{align*}
Therefore, from \eqref{per_estimate} it follows that 
\begin{equation}\label{upper_er_ineq}
P_\phi(E_\lambda, B_r) \le \int_{E_\lambda\cap\p B_r} \phi(\nu_{B_r})d\cH^{n-1} + c_\Phi n\omega_n^{1/n}\,|E_\lambda\cap B_r|^{\frac{n-1}{n}}. 
\end{equation}
Adding $\int_{E_\lambda\cap\p B_r} \phi(\nu_{B_r})d\cH^{n-1} $ to both sides of \eqref{upper_er_ineq}, and using \eqref{Phi_assump} and the isoperimetric inequality  we get 
$$
c_\Phi n\omega_n^{1/n}|E_\lambda\cap B_r|^{\frac{n-1}{n}} \le 4C_\Phi \cH^{n-1}(E_\lambda\cap\p B_r).
$$
Integrating this over $r$ we get the lower volume density estimate in \eqref{volume_dens_two}. 

To get the upper volume density estimate in \eqref{volume_dens_two} we use $\functwo(E_\lambda,E_0,\lambda) \le \functwo(E_\lambda\cup B_r,E_0,\lambda)$ and proceed as above. 

For what concerns the  upper perimeter density estimate in \eqref{perimeter_dens_two} we observe that from \eqref{upper_er_ineq} and \eqref{Phi_assump} it follows that 
\begin{equation*}
2c_\Phi P(E_\lambda, B_r) \le (2C_\Phi + c_\Phi) n\omega_n r^{n-1} 
\end{equation*}
for a.e. $r\in(0,C_3\lambda^{-1/2}).$  Since $r\mapsto P(E_\lambda,B_r)$ is non-decreasing and left-continuous, this inequality holds for all $r.$
Finally the lower perimeter density estimate follows from \eqref{volume_dens_two} and the relative isoperimetric inequality for the ball.
\smallskip

{\it Step 5: Existence of GMM starting from $G$.}  We follow the arguments of \cite{LS:1995,MSS:2016}. Let $\{G(\lambda,k)\}_{\lambda>C_5,k\in\N_0}$ be defined as follows: $G(\lambda,0)=G$ and 
$$
G(\lambda,k)\in\argmin \functwo(\cdot,G(\lambda,k-1),\lambda),\qquad k\ge1.
$$
By Step 1  $G(\lambda,k)$ is well-defined and 
\begin{equation}\label{topmas_pana}
G(\lambda,k)\subseteq D 
\end{equation}
 for all $\lambda>C_5$ and $k\ge0.$ Notice also that  
\begin{equation}\label{increasekkk}
k\in\N_0\mapsto P_\phi(G(\lambda,k)) + \int_{G(\lambda,k)} Hdx\quad \text{is non-increasing.} 
\end{equation}

Given $t>s > 0$ with $t-s<1,$ let $\lambda>\max\{C_5,\frac{5+C_3^2}{t-s},\frac{5}{s}\}$ so that $[\lambda t]-[\lambda s]\ge4,$ $[\lambda s]\ge5$ and $\frac{1}{\lambda|t-s|^{1/2}}< C_3\lambda^{-1/2}.$ By Proposition \ref{prop:funny_estimate} applied with $A=G(\lambda,k-1),$ $r_0=C_3\lambda^{-1/2},$ $\theta:=C_4,$ $\ell:=\frac{1}{\lambda|t-s|^{1/2}}$ and $B=G(\lambda,k),$ and using the bounds  \eqref{Phi_assump} and \eqref{Psi_assump} for anisotropies and mobilities,  
for any $k\in\{[\lambda s]+1,\ldots,[\lambda t]\}$ we get 
$$
\begin{aligned}
|G(\lambda,k-1)\Delta G(\lambda,k)| \le &
\frac{5^n\omega_n }{2C_4 c_\Phi\lambda|t-s|^{1/2}} \,P_\phi(G(\lambda,k-1))\\  
& + \frac{\lambda|t-s|^{1/2}}{2 c_\Psi}  \int_{G(\lambda,k-1)\Delta G(\lambda,k)} d_\psi^{G(\lambda,k-1)}  dx.
\end{aligned}
$$
Therefore,
\begin{equation}\label{bandaman}
\begin{aligned}
|G(\lambda,[\lambda s])\Delta G(\lambda,[\lambda t])|  \le  & \sum\limits_{k=[\lambda s]+1}^{[\lambda t]} |G(\lambda,k-1)\Delta G(\lambda,k)| \\
\le & \frac{5^n\omega_n  }{2C_4c_\Phi \lambda|t-s|^{1/2}} \sum\limits_{k=[\lambda s]+1}^{[\lambda t]} P_\phi(G(\lambda,k-1)) \\
& +   \frac{\lambda|t-s|^{1/2}}{2c_\Psi} \sum\limits_{k=[\lambda s]+1}^{[\lambda t]} \int_{G(\lambda,k-1)\Delta G(\lambda,k)} d_\psi^{G(\lambda,k-1)} dx.
\end{aligned}
\end{equation}
By \eqref{increasekkk},
\begin{align*}
\sum\limits_{k=[\lambda s]+1}^{[\lambda t]} P_\phi(G(\lambda,k-1)) \le &\sum\limits_{k=[\lambda s]+1}^{[\lambda t]} \Big(P_\phi(G(\lambda,k-1)) + \int_{G(\lambda,k-1)} Hdx + \int_{G(\lambda,k-1)} |H|dx\Big)\\
\le &\Big(P_\phi(G) + \int_G Hdx + \int_D|H|dx\Big) \Big([\lambda t] -[\lambda s]\Big)\\
\le & \Big(P_\phi(G) + 2\int_D|H|dx\Big)\, \Big(\lambda (t-s)+1\Big)
\end{align*}
and 
\begin{align*}
&  \lambda \sum\limits_{k=[\lambda s]+1}^{[\lambda t]}  \int_{G(\lambda,k-1)\Delta G(\lambda,k)} d_\psi^{G(\lambda,k-1)} dx \\
\le  &P_\phi(G(\lambda,[\lambda s])) + \int_{G(\lambda,[\lambda s])} Hdx - P_\phi(G(\lambda,[\lambda t])) - \int_{G(\lambda,[\lambda t])} Hdx \\
\le & P_\phi(G) +2\int_D |H|dx,
\end{align*}
therefore, from \eqref{bandaman} we get 
\begin{equation}\label{avvalgilarga_oxshamas}
|G(\lambda,[\lambda s])\Delta G(\lambda,[\lambda t])|  \le  \Big(C_6\,|t-s|^{1/2} + \frac{C_6-\frac{1}{2c_\Psi}}{\lambda|t-s|^{1/2}}\Big)\,\Big(P_\phi(G) + 2\int_D|H|dx\Big),
\end{equation}
where 
\begin{equation}\label{defionition_c6}
C_6:=  \frac{5^n\omega_n}{2C_4c_\Phi}+\frac{1}{2c_\Psi}. 
\end{equation}
Now \eqref{sokmang_otkanlar} and \eqref{rasvoi_dajal} follow from \eqref{avvalgilarga_oxshamas} and \eqref{topmas_pana}, respectively.
\end{proof}

We will use \eqref{perimeter_dens_two}, \eqref{good_bound_radisu} and \eqref{avvalgilarga_oxshamas} in the proof of Theorem \ref{teo:existence_GMM_half}.
 
\section{Improved time H\"older regularity}\label{sec:improve_holder}

In this section we show that when $\phi_i=\phi$ and $\psi_i=\psi$  for any $i=1,\ldots,N+1,$  the time H\"older continuity exponent of GMM for partitions can be improved to $1/2$.
The result follows from the generalization of~\cite{LS:1995} in the previous section (Theorem~\ref{teo:existence_GMM2f}) combined with   with a comparison (Theorem~\ref{teo:comparison} below) between a multiphase flow and   a two-phase flow starting from just one of the phases and its complement. Arguments from our main continuity   result (in Theorem~\ref{teo:existence_GMM}) are needed to reconnect both flows in the limit. 

\begin{theorem}\label{teo:existence_GMM_half}
Let $\Phi=\{\phi,\ldots,\phi\}$ and $\Psi=\{\psi,\ldots,\psi\}$ for some norms $\phi$ and $\psi$ on $\R^n,$ and ${\bf H}\equiv0.$ Then  for any $\cG\in \P_b(N+1)$ and $\cM\in GMM(\func,\cG)$ 
\begin{equation}\label{hulder_est3} 
|\cM(t)\Delta \cM(t')|\le  C_6\,\Per_\Phi(\cG)\,|t-t'|^{1/2},\qquad t,t'>0,\,\,|t-t'|<1,
\end{equation}
where $C_6$ is given in \eqref{defionition_c6}. 
In addition, if $\sum\limits_{j=1}^{N+1} |\overline{G_j}\setminus G_j| = 0,$ then \eqref{hulder_est3} holds for any $t,t'\ge0$ with $|t-t'|<1.$
\end{theorem}

Recall that, by Theorem \ref{teo:existence_GMM}, for any $\cG\in \P_b(N+1),$ $ GMM(\func,\cG)$ is non-empty, each $\cM\in GMM(\func,\cG)$ is locally $1/(n+1)$-H\"older continuous and 
\begin{equation*}
\bigcup\limits_{i=1}^N M_i(t) \subseteq \, \hull{\cG}\qquad \text{for any $t\ge0$}.
\end{equation*}

Besides $\func$ we need to consider also the functional $\functwo$ defined (up to constants) in \eqref{funktwoe} with $H=0,$ i.e., 
\begin{equation}\label{functo2}
\functwo(G,E,\lambda):= P_\phi(G) + \int_{G\Delta E} \dist_\psi(x,\p E)dx. 
\end{equation}

We start with a comparison result: this is the key point of the proof of Theorem \ref{teo:existence_GMM_half} since it allows to compare the evolution of a single phase with the   multiphase case.

\begin{theorem}[\textbf{Discrete comparison multiphase-phase}]\label{teo:comparison}
Let $g_1,\ldots,g_{N+1}\in L_{\loc}^1(\R^n)$  and suppose that $\cA\in\P_b(N+1)$  minimize 
$$
\cE\in \P_b(N+1)\mapsto \sum_{i=1}^{N+1} P_\phi(E_i) + \sum_{i=1}^N \int_{E_i} g_i dx - \int_{E_{N+1}^c} g_{N+1}dx.
$$
Suppose that for $i\in\{1,\dots,N\}$ and $g_i'\in L^1_\loc(\R^n)$, there exists a bounded minimizer $F_i$ of
$$
F\in BV(\R^n;\{0,1\}) \mapsto  P_\phi(F) + \int_F g'_idx, 
$$
and suppose that, given $g_{N+1}'\in L_\loc^1(\R^n),$ there exists a bounded minimizer of
$$
G\in BV(\R^n;\{0,1\}) \mapsto  P_\phi(G) - \int_G g'_{N+1}dx, 
$$
the complement of which we denote by $F_{N+1}.$ If $2g'_i - g_i +  g_j >0$ a.e. in $\R^n$ for all $i,j\in \{1,\ldots,N+1\},$ $i\ne j,$ then  
$$
F_i\subseteq  A_i,\qquad i\in\{1,\ldots,N+1\}.
$$
\end{theorem}

\begin{proof}
Let $i\in\{1,\ldots,N\}.$ By minimality,
\begin{equation}\label{eq:minA}
\begin{aligned}
&\sum_{j=1}^{N+1} P_\phi(A_j) + \sum_{j=1}^N \int_{A_j} g_j dx - \int_{A_{N+1}^c} g_{N+1} dx
\le   P_\phi(A_i\cup F_i) \\
&+ \sum_{j=1, j\neq i}^{N+1} P_\phi(A_j\setminus F_i)  
+ \int_{A_i\cup F_i} g_i dx + \sum_{j=1,\,j\neq i}^N \int_{A_j\setminus F_i} g_j dx - \int_{A_{N+1}^c\cup F_i} g_{N+1} dx
\end{aligned}
\end{equation}
and
\begin{equation}\label{eq:minF}
P_\phi(F_i) + \int_{F_i} g'_idx \le P_\phi(F_i\cap A_i) +  \int_{F_i\cap A_i} g'_i dx.
\end{equation}
Summing \eqref{eq:minA} and twice \eqref{eq:minF}, we obtain
\begin{equation}\label{eq:mainineq}
\begin{aligned}
& P_\phi(F_i) +P_\phi(A_i) + P_\phi(F_i)+ \sum_{j=1,j\neq i}^{N+1 } P_\phi(A_j) \\
& +\int_{F_i\setminus A_i} (g'_i-g_i) dx + \sum_{j=1,j\neq i}^{N+1} \int_{A_j\cap F_i} g_j dx
+ \int_{F_i\setminus A_i} g'_i dx \\ 
\le & P_\phi(F_i\cup A_i) + P_\phi(F_i\cap A_i) + \sum_{j=1,j\neq i}^{N+1}  P_\phi(A_j\setminus F_i)  
+P_\phi(F_i\cap A_i).
\end{aligned}
\end{equation}
Let us show that for any $E\in BV(\R^n;\{0,1\}),$  $\cG\in \P_b(N+1)$ and $i\in\{1,\ldots,N+1\},$  
\begin{equation}\label{eq:compij}
\sum_{j=1,j\neq i}^{N+1} P_\phi(G_j\setminus E)   + P_\phi(G_i\cap E) 
\le P_\phi(E)+ \sum_{j=1,j\neq i}^{N+1} P_\phi(G_j).
\end{equation}
First assume that $\cH^{n-1}(\p^*E \cap \bigcup_{j=1}^{N+1} \p^*G_j)=0.$  In this case by \eqref{ess_intersection} and \eqref{ess_differense},  as well as the inclusion $\partial^* G_i \subset\bigcup_{j=1,j\neq i}^{N+1} \partial^* G_j,$ we obtain
$$
P_\phi(G_j\setminus E) = \int_{E^{(0)} \cap \p^*G_j} \phi(\nu_{G_j})d\cH^{n-1} +  \int_{G_j \cap \p^*E} \phi(\nu_E)d\cH^{n-1} 
$$
and 
\begin{align*}
P_\phi(G_i \cap E) = & \int_{E\cap \p^*G_i} \phi(\nu_{G_i})d\cH^{n-1} +  \int_{G_i \cap \p^*E} \phi(\nu_E)d\cH^{n-1} \\
= & \sum\limits_{j=1,j\ne i}^{N+1} \int_{E\cap \p^*G_j\cap \p^*G_i} \phi(\nu_{G_j})d\cH^{n-1} +  \int_{G_i \cap \p^*E} \phi(\nu_E)d\cH^{n-1},  
\end{align*}
and hence, 
\begin{align*}
&\sum_{j=1,j\neq i}^{N+1} P_\phi(G_j\setminus E)   + P_\phi(G_i\cap E)\\
= & \sum\limits_{j=1,j\ne i}^{N+1}  \Big(\int_{E^{(0)} \cap \p^*G_j} \phi(\nu_{G_j})d\cH^{n-1} + \int_{E\cap \p^*G_j\cap \p^*G_i} \phi(\nu_{G_j})d\cH^{n-1} \Big)\\
& +  \sum\limits_{j=1}^{N+1}  \int_{G_j \cap \p^*E} \phi(\nu_E)d\cH^{n-1} \le  \sum\limits_{j=1,j\ne i}^{N+1} P_\phi(G_j) + P_\phi(E). 
\end{align*}
In the general case we choose a sequence $\{\xi_k\}\subset\R^n$ such that $|\xi_k|\to0$ and $\cH^{n-1}(\p^*(E +\xi_k) \cap \bigcup_{j=1}^{N+1} \p^*G_j)=0$, where 
$
E +\xi_k:= \{x\in\R^n:\,\,x-\xi_k\in E\}.
$ 
By the previous case,
\begin{equation}\label{choram_asbobim}
 \sum_{j=1,\neq i}^{N+1} P_\phi(G_j\setminus (E+\xi_k))   + P_\phi(G_i\cap (E+\xi_k))\le  \sum\limits_{j=1,j\ne i}^{N+1} P_\phi(G_j) + P_\phi(E+\xi_k). 
\end{equation}
Since $P_\phi(E+\xi_k) = P_\phi(E)$ and $\lim\limits_{k\to+\infty} |(E+\xi_k)\Delta E|\to0,$  letting $k\to +\infty$ in \eqref{choram_asbobim} and using the $L^1(\R^n)$-lower semicontinuity of the $\phi$-perimeter we get \eqref{eq:compij}.

Inserting \eqref{eq:compij} with $\cG=\cA$ and $E=F_i$  in \eqref{eq:mainineq} and using \eqref{famfor} we get  
$$
\int_{F_i\setminus A_i} (g'_i-g_i) dx + \sum_{j=1,j\neq i}^{N+1} \int_{A_j\cap F_i} g_j dx
+ \int_{F_i\setminus A_i} g'_i dx \le 0.
$$
Recall that  $F_i\setminus A_i  = \bigcup_{j=1,j\neq i}^{N+1} F_i\cap A_j$ up to a negligible set, thus, 
$$
\sum\limits_{j=1,j\ne i}^{N+1} \int_{A_j\cap F_i} (2g'_i - g_i + g_j) dx \le 0.
$$
By assumption  $2g'_i - g_i +  g_j >0$ a.e., and hence  $F_i\subseteq  A_i$  up to a negligible set. 
The case $i=N+1$ is similar.
\end{proof}

\begin{lemma}\label{lem:sdist_property}
Let $\cG\in \P_b(N+1)$  and set 
$$
g_j(\cdot):= \sdist_\psi(\cdot,\p G_j),\qquad j\in\{1,\ldots,N+1\}.
$$
For $E\subseteq \R^n$ define  $e(\cdot) = \sdist_{\psi}(\cdot,\p E).$ If either $E\subseteq G_i$ for some $i\in\{1,\ldots,N\}$ or $G_{N+1}^c\subseteq E^c,$   then $2e -g_i+g_j\ge 0$ a.e. in $\R^n$ for any $j\in\{1,\ldots,N+1\},$ $j\ne i.$ Similarly, if either $E\strictlyincluded G_i$ for some $i\in\{1,\ldots,N\}$ or $G_{N+1}^c\strictlyincluded E^c,$ then  $2e -g_i+g_j\ge 0$ a.e. in $\R^n$ for any $j\in\{1,\ldots,N+1\},$ $j\ne i.$ 
\end{lemma}

\begin{proof}
Since $E_i^c\subseteq  G_i^c\cup G_j^c,$ the assertion follows from the relation  
$$
A\subseteq  B\qquad \Longrightarrow \qquad \sdist_\psi(\cdot,\p A) \ge \sdist_\psi(\cdot,\p B)\,\,\text{a.e. in $\R^n.$}
$$
\end{proof}

\begin{lemma}\label{lem:minimal_minimizer}
Given $\cA\in\P_b(N+1),$  let $\cA(\lambda)$ minimize $\func(\cdot,\cA,\lambda)$ with ${\bf H}=0.$ For $i\in \{1,\ldots,N+1\}$ let $E\in BV(\R^n;\{0,1\})$ be such that $E\subseteq  A_i;$ in case $i=N+1$ we assume also that $E^c$ is bounded. Then there exists a minimizer $E_i(\lambda)$ of $\functwo(\cdot,E,\lambda)$ such that $E_i(\lambda)\subseteq  A_i(\lambda).$ 
\end{lemma}

\begin{proof}
First we assume that $E=A_i.$
Let $E_1\strictlyincluded E_2\strictlyincluded \ldots\strictlyincluded A_i$ be sets of finite perimeter such that $A_i = \bigcup\limits_k E_k$ and $\sdist(\cdot,\p E_k) \to \sdist(\cdot,\p A_i)$ a.e. as $k\to + \infty.$  Let $E_k(\lambda)$ be a  minimizer of $\functwo(\cdot,E_k,\lambda).$ By \cite{ChMP:2015}, $E_1(\lambda)\subseteq E_2(\lambda) \ldots$ and $ E(\lambda)_*:=\bigcup\limits_{k} E_k(\lambda)$ is the minimal minimizer of $\functwo(\cdot,A_i,\lambda).$ Since $E_k\strictlyincluded A_i,$ by Lemma \ref{lem:sdist_property},
$$
2\sdist(\cdot,\p E_k) - \sdist(\cdot,\p A_i) +\sdist(\cdot,\p A_j)>0\quad \text{a.e. in $\R^n$ for all $j\ne i.$}
$$
Thus, by Theorem \ref{teo:comparison}, $E_k(\lambda)\subseteq  A_i(\lambda).$ Hence, we get $E(\lambda)_*\subseteq  A_i(\lambda).$

In the general case, we consider the minimal minimizer $E(\lambda)_*$ of $\functwo(\cdot,E,\lambda)$ and the minimal minimizer $A_i(\lambda)_*$ of $\functwo(\cdot,A_i,\lambda).$ Since $E\subseteq A_i, $ by \cite{ChMP:2015}, $E(\lambda)_*\subseteq  A_i(\lambda)_*.$ Hence, $E_i(\lambda)=E(\lambda)_*$ satisfies the assertion of the lemma.
\end{proof}

\begin{proof}[Proof of Theorem \ref{teo:existence_GMM_half}]
Given $\cG\in \P_b(N+1)$ define $\{\cG(\lambda,k)\}_{\lambda\ge1,k\in\N_0}$ as follows: $\cG(\lambda,0) = \cG$ and 
$$
\cG(\lambda,k) \in\argmin \func(\cdot,\cG(\lambda,k-1),\lambda),\qquad k\ge1.
$$
Note that the map 
$ 
k\in\N_0\mapsto \Per_\Phi(\cG(\lambda,k))
$ 
is nonincreasing. In particular, 
\begin{equation}\label{kimdir_eshik_ochar}
\Per_\Phi(\cG(\lambda,k)) \le \Per_\Phi(\cG). 
\end{equation}

For any $i\in \{1,\ldots,N+1\}$ and $k\ge0,$ let $\{\twofase{i}{k}{\lambda}{l} \}_{l\ge k}$ be defined as follows: $\twofase{i}{k}{\lambda}{k} :=G_i(\lambda,k)$ and $\twofase{i}{k}{\lambda}{l} $ is the minimal minimizer of $\functwo(\cdot,\twofase{i}{k}{\lambda}{l-1},\lambda)$ for $l>k.$  
Notice that,
according to Step 2 of the proof of Theorem \ref{teo:existence_GMM}, our
actual initial set $\twofase{i}{k}{\lambda}{k}=G_i(\lambda,k)$ satisfies 
the density estimates \eqref{eq:vol.density_est}-\eqref{eq:per.density_est} for all radii $r\le O(1/\lambda)$ and,
according to the proof of Step 4 of Theorem \ref{teo:existence_GMM2f},
all $\twofase{i}{k}{\lambda}{l},$ $l>k,$ satisfy the density estimates \eqref{volume_dens_two}-\eqref{perimeter_dens_two} for all radii $r\le O(1/\lambda^2).$ Moreover, since the initial set $\twofase{i}{k}{\lambda}{k}$ also depends on $\lambda,$ we cannot use the arguments of the $\frac{1}{n+1}$-H\"older continuity up to time $0$ in the proof of Theorem \ref{teo:existence_GMM}.  
\black 

For shortness we call $\{\twofase{i}{k}{\lambda}{l}\}_{l\ge k}$ a discrete solution starting from $\twofase{i}{k}{\lambda}{k}=G_i(\lambda,k).$   
Applying Lemma \ref{lem:minimal_minimizer} inductively one can show that 
$$
\twofase{i}{k}{\lambda}{l} \subseteq  G_i(\lambda,l),\qquad l\ge k.
$$
In particular, 
$$
G_i(\lambda,l) = \Big(\bigcup\limits_{j\ne i} G_j(\lambda,l)\Big)^c = \bigcap\limits_{j\ne i} G_j(\lambda,l)^c \subseteq \bigcap\limits_{j\ne i} \twofase{j}{k}{\lambda}{l} ^c. 
$$
Hence, using $\twofase{i}{k}{\lambda}{k} :=G_i(\lambda,k)$ for all $i=1,\ldots,N+1,$ we get 
$$
G_i(\lambda,l) \setminus G_i(\lambda,k) \subseteq  \Big(\bigcap\limits_{j\ne i} \twofase{j}{k}{\lambda}{l} ^c \Big) \cap \Big(\bigcup\limits_{j\ne i} \twofase{j}{k}{\lambda}{l}  \Big) \subseteq  \bigcup\limits_{j\ne i} \big(\twofase{j}{k}{\lambda}{k} \setminus  \twofase{j}{k}{\lambda}{l} \big).
$$
On the other hand, 
$$
G_i(\lambda,k) \setminus G_i(\lambda,l) =  \twofase{i}{k}{\lambda}{k}  \setminus G_i(\lambda,l) \subseteq \twofase{i}{k}{\lambda}{k}  \setminus \twofase{i}{k}{\lambda}{l}, 
$$
hence, 
\begin{equation}\label{firs_esrt}
|\cG(\lambda,k) \Delta \cG(\lambda,l)| \le \sum\limits_{i=1}^{N+1} |\twofase{i}{k}{\lambda}{k}  \setminus \twofase{i}{k}{\lambda}{l}|, 
\end{equation}
which is the inequality that will allow  us to get the $1/2$-H\"olderianity of GMM.

Fix $i\in\{1,\ldots,N+1\}$ and choose arbitrary $t>s'>s>0.$ Let $\{\twofase{i}{[\lambda s]}{\lambda}{l}\}_{l\ge[\lambda s]}$ be a discrete solution starting from $\twofase{i}{[\lambda s]}{\lambda}{[\lambda s]} =G_i(\lambda,[\lambda s]).$ Then for any $\lambda>\frac{5}{s'-s} +\frac{5}{t-s'}+ \frac{5}{s}$ we have
\begin{equation}\label{kiyiknoma}
\begin{aligned} 
|\twofase{i}{[\lambda s]}{\lambda}{[\lambda s]}) \setminus \twofase{i}{[\lambda s]}{\lambda}{[\lambda t]} |  \le   \sum\limits_{l=[\lambda s]+1}^{[\lambda s']}  |\twofase{i}{[\lambda s]}{\lambda}{l}  \Delta \twofase{i}{[\lambda s]}{\lambda}{l-1}|  \\
+ \sum\limits_{l=[\lambda s']+1}^{[\lambda t]} | \twofase{i}{[\lambda s]}{\lambda}{l} \Delta \twofase{i}{[\lambda s]}{\lambda}{l-1}|=:I_1+I_2. 
\end{aligned} 
\end{equation}
Note that by the choice of $\lambda,$ we have $[\lambda s'] -[\lambda s]\ge4,$ $[\lambda t] -[\lambda s']\ge4$ and $[\lambda s]\ge4.$
According to Step 4 of the proof of Theorem \ref{teo:existence_GMM2f}, $\twofase{i}{[\lambda s]}{\lambda}{l},$ $l\ge[\lambda s]\ge4$ satisfies the uniform lower perimeter density estimate 
$$
C_4 \le \frac{P(\twofase{i}{[\lambda s]}{\lambda}{l},B_r(x))}{r^{n-1}},\qquad x\in\p \twofase{i}{[\lambda s]}{\lambda}{l},\quad r\in(0,C_3\lambda^{-1/2}),
$$
provided $\lambda>C_5.$ Hence, from \eqref{avvalgilarga_oxshamas},
$$
I_2 \le  \Big(C_6  |t-s'|^{1/2} + \frac{C_6-\frac{1}{2c_\Psi}}{\lambda|t-s'|^{1/2}}\Big)\,P_\phi(G_i(\lambda,[\lambda s])).
$$
 Since $\cG(\lambda,[\lambda s])$ minimizes $\func(\cdot,\cG(\lambda,[\lambda s]-1),\lambda),$ by Step 3 of the proof of Theorem \ref{teo:existence_GMM}, see in particular \eqref{youtube_tarmogi} and \eqref{eq:per.density_est},
\begin{equation}\label{daffafaf}
c^\Phi(N,n) \le \frac{P(G_i(\lambda,[\lambda s]),B_r(x))}{r^{n-1}},\quad x\in\p G_i(\lambda,[\lambda s]),\quad r\in\Big(0, \frac{C(n,N,p,\Phi,\Psi)}{\lambda}\Big).
\end{equation}
Because of the presence of $1/\lambda$ (instead of $1/\lambda^{1/2}$) in \eqref{daffafaf}, in general we cannot use \eqref{avvalgilarga_oxshamas}. To estimate $I_1$ we proceed as in the proof of \eqref{eq:muhim_holder} and get 
$$
I_1 \le \Big({\rm C}\,|s'-s|^{\frac{1}{n+1}} + \frac{\tilde {\rm C}}{\lambda |s'-s|^{\frac{n}{n+1}}}\Big)\,P_\phi(G_i(\lambda,[\lambda s])). 
$$
From the estimates for $I_1$ and $I_2,$ and \eqref{firs_esrt},\eqref{kiyiknoma} and \eqref{kimdir_eshik_ochar} we obtain 
\begin{align}
& |\cG(\lambda,[\lambda s]) \Delta \cG(\lambda,[\lambda t])| \le   \Big({\rm C}\,|s'-s|^{\frac{1}{n+1}} + \frac{\tilde {\rm C}}{\lambda |s'-s|^{\frac{n}{n+1}}}\Big)\,\Per_\Phi(\cG(\lambda,[\lambda s]))\no \\
& +\Big(C_6 |t-s'|^{1/2} + \frac{C_6-\frac{1}{2c_\Psi}}{\lambda|t-s'|^{1/2}}\Big)\,\Per_\Phi(\cG(\lambda,[\lambda s])) \label{buni_umr_derlar}\\
\le  & \Big({\rm C}\,|s'-s|^{\frac{1}{n+1}} + \frac{\tilde {\rm C}}{\lambda |s'-s|^{\frac{n}{n+1}}} +  C_6 |t-s'|^{1/2} + \frac{C_6-\frac{1}{2c_\Psi}}{\lambda|t-s'|^{1/2}}\Big)\Per_\Phi(\cG).\no
\end{align}

Now if $\cM\in  GMM(\func,\cG),$  there exists $\lambda_k\to+\infty$ for which 
$$
\lim\limits_{k\to\infty} |\cG(\lambda_k,[\lambda_kt])\Delta \cM(t)|=0 \qquad\text{for any $t\ge0.$}
$$
Thus, from \eqref{buni_umr_derlar} we get 
$$
|\cM(s) \Delta \cM(t)| \le  \Big({\rm C}\,|s'-s|^{\frac{1}{n+1}} + C_6  |t-s'|^{1/2}\Big)\Per_\Phi(\cG).
$$
Since $s'\in(s,t)$ is arbitrary, letting $s'\searrow s$ we get \eqref{hulder_est3} with $C:=C_6.$

Finally, if $\sum\limits_{j=1}^{N+1} |\overline{G_j}\setminus G_j|=0,$ then for any $t>s>0$ and $\cM\in GMM(\func,\cG)$ we have 
$$
|\cM(t) \Delta\cG| \le |\cM(t)\Delta \cM(s)| + |\cM(s)\Delta \cG| \le C_6\Per_\Phi(\cG)|t-s|^{1/2} + {\rm C}s^{\frac{1}{n+1}},
$$
where in the second inequality we used \eqref{hulder_est3} and \eqref{hulder_est}. Now letting $s\searrow 0$  we get 
$$
|\cM(t) \Delta\cG| \le C_6\Per_\Phi(\cG)\,t^{1/2}.
$$
\end{proof}

From Lemma \ref{lem:minimal_minimizer}  we get the following weak comparison property of GMM. 

\begin{theorem}[\textbf{Comparison}]\label{teo:weak_comparison} 
Let $\Phi=\{\phi,\ldots,\phi\}$ and $\Psi=\{\psi,\ldots,\psi\}$ for some norms $\phi$ and $\psi$ on $\R^n,$ and ${\bf H}\equiv0.$ Given $\cG\in\P_b(N+1),$ let $\cM\in GMM(\func,\cG)$ and given $i\in\{1,\ldots,N+1\},$  let $C\in BV(\R^n;\{0,1\})$ be such that $C\subseteq  G_i;$ in case $i=N+1$ we assume also that $C^c$ is bounded. Then there exists   $N \in GMM(\functwo,C)$ such that $N(t)\subseteq  M_i(t)$ for all $t\ge0.$
\end{theorem}

\begin{proof}
Let $\lambda_h\to+\infty$  be such that 
\begin{equation}\label{lone_conver_part}
\lim\limits_{h\to\infty} |\cG(\lambda_h,[\lambda_ht])\Delta \cM(t)|=0\qquad\text{for all $t\ge0,$} 
\end{equation}
where for any $h$ the sequence $\{\cG(\lambda_h,k)\}_{k\in\N_0}$ is defined as: $\cG(\lambda_h,0) = \cG$ and 
$$
\cG(\lambda_h,k) \in\argmin \func(\cdot,\cG(\lambda_h,k-1),\lambda_h),\qquad k\ge1.
$$
Let  $i\in\{1,\ldots,N+1\}$ and $C\in BV(\R^n;\{0,1\})$ be as in the statement. For any $h$ let $\{G(\lambda_h,k)\}_{k\in\N_0}$ be defined as $G(\lambda_h,0)=C$ and 
$G(\lambda_h,k)$ is the minimal minimizer of $\functwo(\cdot,G(\lambda_h,k-1),\lambda),$ $k\ge1$ (see the proof of Lemma \ref{lem:minimal_minimizer} for the definition).   Applying  Lemma \ref{lem:minimal_minimizer}  inductively we get 
\begin{equation}\label{inclusion_principle}
G(\lambda_h,k)\subseteq  G_i(\lambda_h,k)\qquad\text{ for all $ h\ge1$ and $k\ge0.$} 
\end{equation}
Passing to a further (not relabelled) subsequence if necessary, we assume that there exists $N\in GMM(\functwo,C)$ such that 
\begin{equation}\label{minimal_gmm}
\lim\limits_{h\to\infty} |G(\lambda_h,[\lambda_ht])\Delta N(t)|=0\qquad \text{for all $t\ge0.$} 
\end{equation}
By \eqref{lone_conver_part} we have 
$$
\lim\limits_{h\to\infty} |G_i(\lambda_h,[\lambda_ht])\Delta M_i(t)|=0\qquad\text{for all $t\ge0.$} 
$$
Now \eqref{inclusion_principle} and \eqref{minimal_gmm} imply that  $N(t)\subseteq M_i(t)$ for all $t\ge0$ up to a negligible set.  
\end{proof}

\begin{corollary}\label{cor:time_est}
Under the assumptions of Theorem \ref{teo:weak_comparison} let $\cG\in\P_b(N+1)$  and $\cM\in GMM(\func,\cG).$ Let $C_i\subseteq G_i,$ $i\in\{1,\ldots,N\},$ and  $C_{N+1}\supseteq \hull{\cG}$ be convex sets and let $L_i\in GMM(\functwo,C_i),$ $i\in\{1,\ldots,N+1\}.$ Then for any $\cM\in GMM(\func,\cG)$ 
\begin{equation}\label{persist_time}
L_i(t)\ne \emptyset\quad  \Longrightarrow \quad M_i(t) \ne \emptyset,\qquad i=1,\ldots,N,
\end{equation}
and 
\begin{equation}\label{diappear_time}
L_{N+1}(t)= \emptyset\quad  \Longrightarrow \quad M_{N+1}(t)=\emptyset.
\end{equation}

\end{corollary}

\begin{proof}
Recall that anisotropic mean curvature flow with a mobility starting from a bounded convex set $C$ is uniquely defined \cite{BCChN:2005}, coincides with the GMM starting from $C$ and extincts at a finite time $t_C>0.$ By Theorem \ref{teo:weak_comparison}, the $i$-th phase $M_i$ of any  $\cM\in GMM(\func,\cG)$ starting from the $i$-th phase $G_i$ of $\cG$ does not disappear in the time-interval $(0,t_{C_i})$ for any $i\in\{1,\ldots,N\}.$ Analogously, Theorem \ref{teo:weak_comparison} implies that $(N+1)$-th phase of $\cM$ becomes empty, i.e., $\R^n\setminus M_{N+1}(t)=\emptyset$  if $t\ge t_{C_{N+1}}.$ 
\end{proof}

\black

\appendix

\end{document}